\documentclass[letterpaper, 10 pt, journal]{ieeeconf}  

\IEEEoverridecommandlockouts                              

\usepackage{graphicx}
\usepackage{epsfig}
\usepackage{amsmath}

\usepackage{amsthm}

\usepackage{amssymb}
\usepackage[noadjust]{cite}
\usepackage{bm}
\usepackage{subfigure}
\usepackage{acronym}
\usepackage{paralist}
\usepackage{float}
\usepackage{epstopdf}
\usepackage{algorithm}
\usepackage{algpseudocode}
\usepackage{mathtools}
\usepackage{multicol} 
\usepackage{accents}
\newcommand{\subparagraph}{}
\usepackage{titlesec}
\usepackage{amsfonts}
\usepackage{color}

\newtheorem{theorem}{Theorem}
\newtheorem{lemma}{Lemma}

\newtheorem{remark}{Remark}
\newtheorem{assume}{Assumption}

\newcommand{\R}{\mathbb R}

\newcommand{\N}{\mathcal{N}}

\newcommand{\V}{\mathcal{V}}

\newcommand{\A}{\mathcal{A}}

\newcommand{\eps}{\epsilon}

\newcommand{\Z}{\mathbb{Z}}

\newcommand{\bmx}[1]{\begin{bmatrix}#1\end{bmatrix}} 
\newcommand{\bkt}[1]{\left[#1\right]} 
\newcommand{\pth}[1]{\left(#1\right)} 
\newcommand{\brc}[1]{\left \{#1\right \}} 
\newcommand{\nrm}[1]{\left \lVert#1\right \rVert} 

\newcommand{\bmxs}[1]{\begin{bsmallmatrix}#1\end{bsmallmatrix}}

\DeclarePairedDelimiter{\ceil}{\lceil}{\rceil}
\DeclarePairedDelimiter{\floor}{\lfloor}{\rfloor}
\DeclarePairedDelimiter{\abs}{\lvert}{\rvert}

\newcommand{\rarr}{\rightarrow} 

\makeatletter
\let\oldceil\ceil
\def\ceil{\@ifstar{\oldceil}{\oldceil*}}
\makeatother
\makeatletter
\let\oldfloor\floor
\def\floor{\@ifstar{\oldfloor}{\oldfloor*}}
\makeatother
\makeatletter
\let\oldnorm\norm
\def\norm{\@ifstar{\oldnorm}{\oldnorm*}}
\makeatother
\makeatletter
\let\oldabs\abs
\def\abs{\@ifstar{\oldabs}{\oldabs*}}
\makeatother

\newcommand{\LSE}{\textup{LSE}}

\newcommand{\overbar}[1]{\mkern 1.5mu\overline{\mkern-1.5mu#1\mkern-1.5mu}\mkern 1.5mu}

\newcommand{\figsize}{0.32}

\newtheorem{problem}{Problem}

\makeatletter
\let\NAT@parse\undefined
\makeatother
\usepackage{hyperref}
\hypersetup{colorlinks,
            pageanchor=false,
            citecolor=red,
            linkcolor=blue
            }

\title{\LARGE \bf Adversarial Resilience for Sampled-Data Systems under High-Relative-Degree Safety Constraints
}

\author{James Usevitch and Dimitra Panagou
\thanks{*The authors wish to acknowledge the technical and financial support of the Automotive Research Center (ARC) in accordance with Cooperative Agreement W56HZV-19-2-0001 U.S. Army CCDC Ground Vehicle Systems Center (GVSC) Warren, MI.}
\thanks{James Usevitch, and Dimitra Panagou are with the Aerospace Engineering Department at the University of Michigan, Ann Arbor.
        {\tt\small \{usevitch, dpanagou\}@umich.edu}}%
}

\begin{document}

\maketitle
\thispagestyle{empty}
\pagestyle{empty}

\begin{abstract}

Control barrier functions (CBFs) have recently become a powerful method for rendering desired safe sets forward invariant in single- and multi-agent systems. In the multi-agent case, prior literature has considered scenarios where all agents cooperate to ensure that the corresponding set remains invariant. 
However, these works do not consider scenarios where a subset of the agents are behaving adversarially with the intent to violate safety bounds. In addition, prior results on multi-agent CBFs typically assume that control inputs are continuous and do not consider sampled-data dynamics.
This paper presents a framework for normally-behaving agents in a multi-agent system with heterogeneous control-affine, sampled-data dynamics to render a safe set forward invariant in the presence of adversarial agents.
The proposed approach considers several aspects of practical control systems including input constraints, clock asynchrony and disturbances, and distributed calculation of control inputs. Our approach also considers functions describing safe sets having high relative degree with respect to system dynamics.
The efficacy of these results are demonstrated through simulations.
\end{abstract}

\section{Introduction}

Guaranteeing the safety of autonomous systems is a critical challenge in modern control theory. 
Safety is frequently modeled by defining a safe subset of the state space for a given system and generating control inputs that render this subset forward invariant. 
Control barrier function (CBF) methods \cite{ames2019control, srinivasan2018control, glotfelter2017nonsmooth, garg2019control} that leverage quadratic programming (QP) techniques have risen as a powerful framework for establishing forward invariance of a safe set. Both single-agent \cite{ames2016control, hsu2015control, xiao2019control, cortez2019control} and multi-agent systems \cite{li2018formally, wang2017safety, glotfelter2017nonsmooth, glotfelter2018boolean, guerrero2019realization} have been considered, where agents have control-affine dynamics. Multi-agent CBF techniques have been applied to a variety of settings including collision avoidance for quadrotors \cite{wang2017safe} and mobile robots \cite{pickem2017robotarium}, accomplishing spatiotemporal tasks \cite{lindemann2019control}, 
forming or maintaining network communication topologies between mobile agents \cite{guerrero2019realization}, and more.

Prior work on multi-agent CBF methods typically assumes that all agents apply the nominally specified control law. This assumption does not encompass faulty or adversarial behavior of agents within the system. In particular, adversarial agents may apply control laws specifically crafted in an attempt to violate set invariance conditions within given control constraints.
Much prior and recent work has considered the accomplishment of control objectives in the presence of faulty or adversarial agents \cite{mitchell2005time, isaacs1999differential, park2017fault, saulnier2017resilient, usevitch2018finite, usevitch2019resilient, usevitch2018resilient}. However, to the authors' best knowledge no prior work using CBF methods have considered the presence of adversarial agents with respect to control actions. CBFs are used in \cite{guerrero2019realization} to construct resilient network communication topologies in finite time; however, all agents are assumed to apply the nominal CBF-based controller without any adversarial misbehavior with respect to control actions.

In addition, the majority of prior work involving CBF methods considers a continuous-time system with continuous inputs. Practical systems are often more appropriately modeled using sampled-data dynamics, where state measurements and control inputs remain constant between sampling times. Notable studies that have explicitly considered the effects of sampling in CBF methods include \cite{cortez2019control, singletary2020control}. However, these papers do not consider multi-agent systems and do not consider the presence of faulty or adversarial agents.
Many systems also consider a CBF having high relative degree with respect to agents' dynamics, where the control input of the agents does not appear in the expression for the first derivative of the function whose sublevel or superlevel sets describe the safe set (e.g., systems with double-integrator dynamics). Methods to apply CBF set-invariance methods to such systems have been presented in prior literature \cite{nguyen2016exponential,xiao2019control}; however these methods do not consider sampled-data dynamics and do not consider the presence of adversarial agents.

In this paper, we present a framework for guaranteeing forward invariance of sets in sampled-data multi-agent systems in the presence of adversarial agents. This framework considers a class of functions describing safe sets that have high relative degree with respect to (w.r.t.) the system dynamics, where the control inputs of the agents do not appear for one or more time derivatives of the safe-set function. Unlike prior work, this paper simultaneously considers multi-agent systems, asynchronous sampling times with clock disturbances, the presence of adversarially behaving agents and functions describing safe sets that have high relative degree w.r.t. the system dynamics.
Our specific contributions are as follows:
\begin{itemize}
    \item We present a method under which a set of normally-behaving agents in a system with sampled-data dynamics can collaboratively render a safe set forward invariant despite the actions of adversarial agents. Our analysis considers asychronous sampling times and distributed calculation of agents' control inputs.
    \item We present a method under which a system of normally-behaving agents with sampled-data dynamics can render a safe set forward invariant in the presence of adversarial agents when the safe set is described by a function with high relative degree with respect to agents' dynamics.
\end{itemize}

Part of this work was previously submitted as a conference paper \cite{usevitch2021adversarially}. The differences between the conference version and this work are as follows:
\begin{itemize}
    \item We include several proofs which were omitted from the conference version due to space constraints.
    \item We extend the results of the conference version \cite{usevitch2021adversarially} to consider functions describing safe sets having high relative degree with respect to the system dynamics.
    \item We present additional simulations to demonstrate the efficacy of our approach.
\end{itemize}

The organization of this paper is as follows: Section \ref{sec:notation} gives the notation and problem formulation, 
Section \ref{sec:mainresults} presents the main results for systems with a relative degree of one are presented, Section \ref{sec:highdegree} presents the main results for functions describing the safe set having high relative degree w.r.t. the system dynamics, Section \ref{sec:simulations} presents simulations demonstrating this paper's results, and Section \ref{sec:conclusion} gives a brief conclusion.


\section{Notation and Problem Formulation}
\label{sec:notation}

The nonnegative and strictly positive integers are denoted $\mathbb{Z}_{\geq 0}$ and $\mathbb{Z}_{>0}$, respectively.
We use the notation $h\in \mathcal C^{1,1}_{loc}$ to denote a continuously differentiable function $h$ whose gradient $\nabla h$ is locally Lipschitz continuous.
Let $x_i \in \R^{n_i}$, $n_i \in \Z_{\geq 1}$ for $i=1,\ldots,N$ be a set of vectors, and let $\bar{n} = \sum_{i=1}^N n_i$. We let $\vec{x} = \bmx{x_1^T,\ldots,x_N^T}^T$ denote the vector concatenating all $x_i$ vectors.
The partial Lie derivative of a function $f(\vec{x})$ with respect to $x_i$ is denoted $L_{f} h^{x_i}(\vec{x}) = \frac{\partial h(\vec{x})}{\partial x_i} f(\vec{x})$.
The $n$-ary Cartesian product of sets $S_1, \ldots, S_N$ is denoted $\bigtimes_{i=1}^N S_i = S_1 \times \ldots \times S_N$.
The Minkowski sum of sets $S_1$, $S_2$ is denoted $S_1 \oplus S_2$. The open and closed norm balls of radius $\eps > 0$ centered at $\vec{x} \in \R^n$ are respectively denoted $B(\vec{x}, \eps)$, $\overbar{B}(\vec{x}, \eps)$. The boundary and interior of a set $S \subset \R^n$ are denoted $\partial S$ and $\textup{int}(S)$, respectively.

\subsection{Problem Formulation}
\label{sec:problem}

Consider a group of $N \in \Z_{>0}$ agents, with the set of agents denoted by $\V$ and each agent indexed $\{1,\ldots,N\}$.
Each agent $i \in \V$ has the state $x_i \in \R^{n_i}$, $n_i \in \Z_{>0}$ and input $u_i \in \R^{m_i}$, $m_i \in \Z_{>0}$.
The system and input vectors $\vec{x}, \vec{u}$, respectively, denote the vectors that concatenate all agents' states and inputs, respectively, as $\vec{x} = \bmx{x_1^T,\ldots, x_N^T}^T,\ \vec{x} \in \R^{\bar{n}}$ and $	\vec{u} = \bmx{u_1^T,\ldots, u_N^T},\ \vec{u} \in \R^{\bar{m}}$, $\bar{n} = \sum_{i=1}^N n_i,\ \bar{m} = \sum_{i=1}^N m_i$.
Agents receive knowledge of the system state $\vec{x}$ in a sampled-data fashion; i.e., each agent $i \in \V$ has knowledge of $\vec{x}(\cdot)$ only at times
$\mathcal{T}_i = \{t_i^0, t_i^1, t_i^2, \ldots\}$,
where $t_i^k$ represents agent $i$'s $k$th sampling time, with $t_i^{k+1} > t_i^{k}$ $\forall k \in \Z_{\geq 0}$.
In addition, at each $t_i^k \in \mathcal{T}_i$ the agent $i$ applies a zero-order hold (ZOH) control input $u(t_i^k)$ that is constant on the time interval $t \in [t_i^k, t_i^{k+1})$.
For brevity, we denote $x_i^{k_i} = x_i(t_i^k)$ and $u_i^{k_i} = u_i(t_i^k)$.
The sampled-data dynamics of each agent $i \in \V$ under its ZOH controller on each interval $t \in [t_i^{k}, t_i^{k+1})$ is as follows:
\begin{align}
\begin{aligned}
\label{eq:generalsystem}
	\dot{x}_i(t) &= f_i(x_i(t)) + g_i(x_i(t)) u_i(t_i^k) + \phi_i(t). \\
\end{aligned}
\end{align}
The functions $f_i$, $g_i$ may differ among agents, but are all locally Lipschitz on their respective domains $\R^{n_i}$.
Note that under these definitions for any $i \in \V$ there exists a matrix $C_i \in \R^{n_i} \times \R^{\bar{n}}$ such that $x_i = C_i \vec{x}$. We abuse notation by sometimes writing an expression $f(x_i)$ as $f(\vec{x})$.
The functions $\phi_i: \R \rarr \R^{n_i}$, $i \in \V$,
are locally Lipschitz in $t$ and
model disturbances to the system \eqref{eq:generalsystem}. 
Each $\phi_i$ is bounded as per the following assumption:
\begin{assume}
For all $i \in \V$, the disturbances $\phi_i(t)$ satisfy $\nrm{\phi_i(t)} \leq \phi_{i}^{\max} \in \R_{\geq 0},\ \forall t \geq 0$.
\end{assume}
Since each control input $u_i(\cdot)$ is piecewise constant, the existence and uniqueness of solutions to \eqref{eq:generalsystem} are guaranteed by Carath\'eodory's theorem \cite[Sec. 2.2]{grune2017nonlinear}.

Each agent $i \in \V$ has control input constraints that
are represented by a nonempty, convex, compact polytope, i.e. $u_i \in \mathcal{U}_i(x_i) = \{u \in \R^{m_i} : A_i(x_i) u \leq b_i(x_i)\}$, where the functions $A_i :\R^{n_i} \rarr \R^{q_i \times m_i}$, $ b_i: \R^{n_i} \rarr \R^{q_i}$ are locally Lipschitz on their respective domains. Representation of control input constraints as polytopes is common in prior literature \cite{ames2016control, garg2019control, garg2019prescribed}.
Similar to prior work, it is assumed there exists a nominal control law $\vec{u}_{\textup{nom}}(\cdot)$ that the system computes in order to accomplish some objective \cite{ames2019control}. Examples of such a $\vec{u}_{\textup{nom}}$ might include a feedback control law to track a time-varying trajectory or to converge to a goal set. The nominal control law is designed without any safety consideration, and therefore it is desired to minimally modify $\vec{u}_\textup{nom}$ in order to render a safe set $S \subset \R^{\bar{n}}$ forward invariant under the dynamics \eqref{eq:generalsystem}. The set $S$ is defined as the sublevel sets of a function $h : \R^{\bar{n}} \rarr \R$, $h \in C^{1,1}_{loc}$ as follows:
\begin{align}
    \begin{aligned}
\label{eq:setdefinition}
    S = \{\vec{x} \in \R^{\bar{n}} : h(\vec{x}) \leq 0\}, \\
    \partial S = \{\vec{x} \in \R^{\bar{n}} : h(\vec{x}) = 0\}, \\
    \text{int}(S) = \{\vec{x} \in \R^{\bar{n}} : h(\vec{x}) < 0\}.
\end{aligned}
\end{align}
\begin{assume}
\label{assume:Scompact}
The set $S$ is compact.
\end{assume}
\begin{assume}
\label{assume:Ui}
For all $i \in \V$ and $\forall \vec{x} \in S$, the interior of $\mathcal{U}_i(\vec{x})$ is nonempty and $\mathcal{U}_i(\vec{x})$ is uniformly compact near $\vec{x}$. 
\end{assume}
\begin{remark}
Note that the conditions for Assumption \ref{assume:Ui} are trivially satisfied when $A_i$, $b_i$ are constant and the interior set $\{u \in \R^{m_i} : A_i u < b_i \}$ is nonempty.
For a specific example satisfying Assumption \ref{assume:Ui} when $\mathcal{U}_i(\cdot)$ is not constant, see \eqref{eq:IOcontrolbounds} in Section \ref{sec:simulations} of this paper.
\end{remark}

We will refer to functions describing safe sets as simply ``safe set functions" for brevity.
For multi-agent systems that apply continuous controllers $u_i(t)$ to the dynamics \eqref{eq:generalsystem}, forward invariance can be collaboratively guaranteed by satisfying the sufficient condition $\dot{h}(\vec{x}(t)) \leq -\alpha(h(\vec{x}(t)))$ based on Nagumo's theorem \cite{nagumo1942lage}, where $\alpha(\cdot)$ is an extended class-$\mathcal{K}_\infty$ function and locally Lipschitz on $\R$. The dependence of $\vec{x}(t)$ on $t$ will be omitted for brevity. For the multi-agent system \eqref{eq:generalsystem}, expanding the term $\dot{h}(\vec{x})$ yields
\begin{align}
\label{eq:Nagumosum}
    \sum_{i \in \V} \pth{L_{f_i} h^{x_i}(\vec{x}) + L_{g_i} h^{x_i}(\vec{x}) u_i + L_{\phi_i} h^{x_i}(\vec{x})} \leq -\alpha(h(\vec{x})),
\end{align}
where the partial Lie derivative notation $L_{f_i} h^{x_i}(\vec{x})$ is defined at the beginning of Section \ref{sec:notation}.
When all agents behave normally, methods exist for agents to locally solve for appropriate local control inputs that together satisfy the condition in \eqref{eq:Nagumosum} (e.g. \cite{lindemann2019decentralized}).

In contrast to prior work, this paper considers systems containing agents that exhibit adversarial behavior. 
More specifically, this paper considers a subset of agents $\A \subset \V$ that apply the following control input for all sampling times $t_j^k$, $k \in \Z_{\geq 0}$, $j \in \A$:
\begin{align}
\label{eq:adversarialu}
    \hspace{-4pt}u_j^{\max}(\vec{x}^{k_j}) = \arg\max_{u \in \mathcal{U}_j} \bkt{L_{f_j}h^{x_j}(\vec{x}^{k_j}) + L_{g_j} h^{x_j}(\vec{x}^{k_j}) u}.
\end{align}
The agents in $\A$ are called \emph{adversarial}.
\begin{remark}
The control input \eqref{eq:adversarialu} models adversarial intent in the sense that \eqref{eq:adversarialu} maximizes agent $j$'s control input contribution to the left-hand side (LHS) of \eqref{eq:Nagumosum}, i.e., the term $L_{g_j}h^{x_j}(\vec{x})u_j$. Violating the inequality in \eqref{eq:Nagumosum} removes the forward invariance guarantee for the safe set $S$, and therefore the control law \eqref{eq:adversarialu} represents an adversarial agent's maximum instantaneous control effort towards violating system safety.
\end{remark}
Agents that are not adversarial are called \emph{normal}. The set of normal agents is denoted $\N = \V \backslash \A$.
Dividing the left-hand side (LHS) of \eqref{eq:Nagumosum} into normal and adversarial parts yields the following sufficient condition for set invariance in the presence of adversaries:
\begin{align}
    &\sum_{j \in \mathcal{A}} \pth{L_{f_j} h^{x_j}(\vec{x}) + L_{g_j} h^{x_j}(\vec{x}) u_j^{\max} + L_{\phi_j} h^{x_j}(\vec{x})} + \label{eq:Nagumo_adversarial}\\
    &\sum_{i \in \mathcal{N}} \pth{L_{f_i} h^{x_i}(\vec{x}) + L_{g_i} h^{x_i}(\vec{x}) u_i + L_{\phi_i} h^{x_i}(\vec{x})} \leq -\alpha(h(\vec{x})). \nonumber
\end{align}
Again, the equation \eqref{eq:Nagumo_adversarial} being satisfied for all $t \geq 0$ is equivalent to $\dot{h}(\vec{x}(t)) \leq \alpha(h(\vec{x}(t)))$ being satisfied for all $t \geq 0$ which implies forward invariance of the set $S$.
The form of \eqref{eq:Nagumo_adversarial} reflects sampled-data adversarial agents seeking to violate the set invariance condition in \eqref{eq:Nagumosum} by maximizing their individual contributions to the LHS sum.
The problem considered in this paper is for the normal agents to compute control inputs that render the set $S$ forward invariant using the sufficient condition in \eqref{eq:Nagumo_adversarial} despite the worst-case behavior of the adversarial agents in $\A$.

\begin{problem}
Determine control inputs for the normal agents $i \in \V$ which render the set $S$ forward invariant under the perturbed sampled-data dynamics \eqref{eq:generalsystem} in the presence of a set of worst-case adversarial agents $\A$.
\end{problem}

\begin{remark}
\label{remark:uncontrollable}
Since adversarial agents' states are generally modeled as being uncontrollable under the nominal system control law, 
the function $h(\vec{x})$ can be defined to consider only the safety of normal agents.
\end{remark}

\begin{remark}
This paper assumes the identities of the adversarial agents are known to the normal agents. Methods for identifying misbehavior are beyond the scope of this paper.
\end{remark}

\section{Safe Set Functions with Relative Degree 1}
\label{sec:mainresults}

We first present results for safe set functions $h$ where the control inputs $u_i$ for all agents appear simultaneously in the expression for the first time derivative $\dot{h}(\vec{x}(t))$. Such functions are said to have relative degree 1 with respect to the system dynamics \eqref{eq:generalsystem}.

\subsection{Preliminaries}
The results of this subsection will be needed for our later analysis.
The minimum and maximum value functions $\gamma_i^{\min}(\cdot)$, $\gamma_i^{\max}(\cdot)$ for $i \in \V$ are defined as follows:
\begin{align}
\begin{aligned}
    \gamma_i^{\min}(\vec{x}) &= \min_{u_i \in \mathcal{U}_i} \bkt{L_{f_i}h^{x_i}(\vec{x}) + L_{g_i}h^{x_i}(\vec{x}) u_i}, \\
    \gamma_i^{\max}(\vec{x}) &= \max_{u_i \in \mathcal{U}_i} \bkt{L_{f_i}h^{x_i}(\vec{x}) + L_{g_i}h^{x_i}(\vec{x}) u_i}. \label{eq:vLPs}
    \end{aligned}
\end{align}
Each $\gamma_i^{\min}(\vec{x})$ and $\gamma_i^{\max}(\vec{x})$ can be calculated by solving a parametric linear program
\begin{alignat}{3}
\begin{aligned}
\label{eq:LPvminmax}
		&\underset{u_i \in \R^{m_i}}{\min} & & c(\vec{x})^T u_i & \text{s.t.} & &   A_i(\vec{x}) u_i \leq b_i(\vec{x}),\\
\end{aligned}
\end{alignat}
where the vector $c(\vec{x})^T = L_{g_i}h^{x_i}(\vec{x})$ when calculating 
$\gamma_i^{\min}$ and $c(\vec{x})^T = -L_{g_i}h^{x_i}(\vec{x})$ when calculating $\gamma_i^{\max}$.
Note that \eqref{eq:LPvminmax} is feasible for all $\vec{x} \in S$ under Assumption \ref{assume:Ui}. 
For an adversarial $j \in \A$, the function $\gamma_j^{\max}(\cdot)$ represents the bound on the worst-case contribution of $j$ to the sum on the LHS of \eqref{eq:Nagumo_adversarial}. Similarly, the function $\gamma_i^{\min}(\cdot)$ for a normal agent $i \in \N$ represents the bound on agent $i$'s best control effort towards minimizing the LHS of \eqref{eq:Nagumo_adversarial}.

\begin{remark}
\label{rem:gammaisdabomb}
Note that for any $j \in \A$, for all $u_j \in \mathcal{U}_j$ it holds that
\begin{align}
    L_{f_j}h^{x_j}(\vec{x}) + L_{g_j} h^{x_j}(\vec{x}) u_j \leq \gamma_j^{\max}(\vec{x}),\quad \forall \vec{x} \in \R^{\bar{n}}.
\end{align}
Due to this property, it will be demonstrated later in this paper that the results obtained by considering $\gamma_i^{\max}$ will hold for any $u_j \in \mathcal{U}_j$ for all $j \in \A$.
\end{remark}

The following result presents a sufficient condition under which $\gamma_i^{\min}(\cdot)$ and $\gamma_i^{\max}(\cdot)$ are locally Lipschitz on the set $S$.
\begin{lemma}
\label{lem:vLips}
If the interior of $\mathcal{U}_i(\vec{x})$ is nonempty for all $\vec{x} \in S$ and $\mathcal{U}_i(\vec{x})$ is uniformly compact near $\vec{x}$ for all $\vec{x} \in S$, then the functions $\gamma_i^{\min}(\cdot)$ and $\gamma_i^{\max}(\cdot)$ defined by \eqref{eq:vLPs} are locally Lipschitz on $S$.
\end{lemma}

\begin{proof}
The proofs for $\gamma_i^{\min}(\cdot)$ and $\gamma_i^{\max}(\cdot)$ are identical except for trivially changing the sign of the objective function; therefore only the proof for $\gamma_i^{\min}(\cdot)$ is given.
Define the set of optimal points
\begin{align*}
P_i(\vec{x}) = \brc{u_i^* : u_i^* = \arg \min_{u \in \mathcal{U}_i} L_{f_i}h^{x_i}(\vec{x}) + L_{g_i}h^{x_i}(\vec{x})u }.
\end{align*}
The result in \cite[Theorem 5.1]{gauvin1982differential} states that if $\mathcal{U}_i(\vec{x})$ is nonempty and uniformly compact near $\vec{x} \in \R^{\bar{n}}$ and if the Mangasarian-Fromovitz (M-F) conditions 
hold at each $u_i^* \in P_i(\vec{x})$, then 
$\gamma_i^{\min}(\cdot)$ is locally Lipschitz near $\vec{x}$ (see \cite{gauvin1982differential} for the definition of the M-F conditions). The first two conditions hold by assumption, and so we next prove that the M-F conditions hold at each $u_i^* \in P(\vec{x})$. Let $A_{i,j}(\cdot)$ denote the $j$th row of $A_i(\cdot)$ and $b_{i,j}(\cdot)$ denote the $j$th entry of $b_i(\cdot)$.

Consider any $\vec{x} \in S$ and $u_i^* \in P_i(\vec{x})$. Denote $J_i(\vec{x}) = \brc{j \in \{1,\ldots, q_i\} : A_{i,j}(\vec{x}) u_i^* - b_{i,j}(\vec{x}) = 0}$ as the set of constraint indices where equality holds at $u_i^*$.
Note that by definition of $J_i(\vec{x})$, for all $j' \not\in J_i(\vec{x})$ it holds that $A_{i,j'}(\vec{x}) < 0$.
The interior $\textup{int}\pth{\mathcal{U}_i(\vec{x})}$ being nonempty and convex implies there exists an $r \in \R^{m_i}$ such that for all $j \in J_i(\vec{x})$,
\begin{align}
\label{eq:Aijbijless}
    &A_{i,j}(\vec{x})\pth{u_i^* + r} - b_{i,j}(\vec{x}) < 0, \nonumber\\
    &\implies A_{i,j}(\vec{x})r < b_{i,j}(\vec{x}) - A_{i,j}(\vec{x})u_i^* = 0.
\end{align}

This implies that there exists an $r$ such that $A_i(\vec{x})r < 0$. The point $u_i^*$ is therefore M-F regular. Since this holds for any $u_i^* \in P_i(\vec{x})$ and $\forall \vec{x} \in S$, by \cite[Theorem 5.1]{gauvin1982differential} it holds that $\gamma_i^{\min}(\cdot)$ is locally Lipschitz on $S$.
\end{proof}
We briefly emphasize the difference between the min / max \textbf{\emph{value}} functions $\gamma_i^{\min}, \gamma_i^{\max}$ in \eqref{eq:vLPs} and the min / max \textbf{\emph{point}} functions defined as
\begin{align}
    u_i^{\min}(\vec{x}) &= \underset{u_i \in \mathcal{U}_i}{\arg\min} \bkt{L_{f_i}h^{x_i}(\vec{x}) + L_{g_i}h^{x_i}(\vec{x}) u_i}, \label{eq:umin}\\
    u_i^{\max}(\vec{x}) &= \underset{u_i \in \mathcal{U}_i}{\arg\max} \bkt{L_{f_i}h^{x_i}(\vec{x}) + L_{g_i}h^{x_i}(\vec{x}) u_i}. \label{eq:umax}
\end{align}
In words, $u_i^{\min}$ and $u_i^{\max}$ represent the control actions 
such that, respectively, $\gamma_i^{\min}(\vec{x}) = L_{f_i}h^{x_i}(\vec{x}) + L_{g_i}h^{x_i}(\vec{x}) u_i^{\min}$ and  $\gamma_i^{\max}(\vec{x}) = L_{f_i}h^{x_i}(\vec{x}) + L_{g_i}h^{x_i}(\vec{x}) u_i^{\max}$.
Although the min / max \textbf{\emph{value}} functions $\gamma_i^{\min}(\cdot), \gamma_i^{\max}(\cdot)$ are locally Lipschitz under the conditions of Lemma \ref{lem:vLips} and \cite{gauvin1982differential}, the min / max \textbf{\emph{point}} functions $u_i^{\min}$ and $u_i^{\max}$ may not be locally Lipschitz in general.\footnote{We re-emphasize however that when \eqref{eq:umax} is applied in a ZOH manner, existence and uniqueness of solutions to \eqref{eq:generalsystem} is guaranteed  by Carath\'eodory's theorem \cite[Sec. 2.2]{grune2017nonlinear}.}

The following Lemma will also be needed for our later analysis, and is based on results in \cite{cortez2019control}, \cite[Thm. 3.4]{khalil2002nonlinear}. It establishes an upper bound on the difference between the sampled state $\vec{x}^{k_i}$ and the state $\vec{x}(t)$ on the time interval $t \in [t_i^k, t_i^{k} + \Gamma)$, $\Gamma \geq 0$.
\begin{lemma}
\label{lem:xdiffbound}
For any $\Gamma \geq 0$, there exists a $\mu \geq 0$, $L' > 0$ such that the following holds:
\begin{align*}
    \nrm{\vec{x}(t) - \vec{x}^{k_i}} \leq \frac{\mu}{L'}\pth{e^{L'\Gamma} -1}\ \forall t \in [t_i^k, t_i^k + \Gamma).
\end{align*}
\end{lemma}


\begin{proof}

Using the same method as \cite[Thm. 3.4]{khalil2002nonlinear}, define the functions
\begin{align}
    \mathfrak{f}(t,\vec{x}) =& 0, \\
    \mathfrak{g}(t,\vec{x}) =& \bmx{ f_1(x_1) + g_1(x_1) u_1(t) + \phi_1(t) \\ \vdots \\ f_N(x_N ) + g_N(x_N) u_N(t) + \phi_N(t)}
\end{align}
Next, observe that
{
\medmuskip=0mu
\thinmuskip=0mu
\thickmuskip=0mu
\begin{align*}
    \frac{d}{dt} \vec{x}^{k_i}\ &=\ 0\ =\ \mathfrak{f}(t,\vec{x}^{k_i}), \\
    \frac{d}{dt} \vec{x}(t)\  &=\ \mathfrak{f}(t, \vec{x}) + \mathfrak{g}(t,\vec{x}). 
\end{align*}
}
Observe that $S$ is compact by Assumption \ref{assume:Scompact}, each $f_i$, $g_i$ is locally Lipschitz, and each $\phi_i(t)$ is locally Lipschitz with $\nrm{\phi_i(t)} \leq \phi_i^{\max}$. In addition, by Assumption \ref{assume:Ui} there exists an upper bound $u_M \in \R$ such that $\nrm{u_l} \leq u_M$. Therefore there exists $\mu \in \R \geq 0$ such that
\begin{align}
    \sup_{\vec{x} \in S} \nrm{\bmx{ f_1(x_1) + g_1(x_1) u_1 + \phi_1(t) \\ \vdots \\ f_N(x_N) + g_N(x_N) u_N + \phi_N(t)}} \leq \mu.
\end{align}
Note that for $t = t_i^k$ we have $\nrm{\vec{x}(t) - \vec{x}^{k_i}(t)} = 0$.
Therefore by \cite[Thm. 3.4]{khalil2002nonlinear}, it holds that
\begin{align}
    \nrm{\vec{x} - \vec{x}^{k_i}} \leq \frac{\mu}{{L'}}\pth{e^{L'(t-t_i^k)}-1}\ \forall t \in [t_i^k, t_i^{k} + \Gamma),
\end{align}
where $L' \in \R_{>0}$ is any strictly positive constant.
\end{proof}
For brevity, we define the function $\eps : \R \times \R \times \R_{>0} \rarr \R$ as
\begin{align}
\label{eq:epsT}
\eps(\Gamma, \mu, L') = \frac{\mu}{{L'}}\pth{e^{L'\Gamma}-1}.
\end{align}
For fixed $\mu, L'$, we abuse notation by writing $\eps(\Gamma)$ as a function of $\Gamma$ only.
It can be shown that for fixed $\mu, L'$,  $\eps(\cdot)$ is a class-$\mathcal{K}$ function in $\Gamma$.

\subsection{Synchronous Sampling Times}

To facilitate the presentation of the main results, we first consider the case where all agents in the system have synchronous sampling times with a period of $\Gamma > 0$, i.e. $\mathcal{T}_i = \{k\Gamma : k \in \Z_{\geq 0}\}$ $\forall i \in \N$. 
This assumption is later relaxed to consider agents with asynchronous, nonidentical sampling times.
The Cartesian product of the admissible controls for all normal agents is denoted 
$\mathcal{U}_{\N} = \bigtimes_{i \in \N} \mathcal{U}_i$.
Under Assumption \ref{assume:Ui}, each $\mathcal{U}_i(\vec{x})$ being uniformly compact near all $\vec{x} \in S$ implies that $\mathcal{U}_{\N}$ is also uniformly compact near all $\vec{x} \in S$. We will denote $\vec{u}_{\N} \in \mathcal{U}_{\N}$ as the vector containing only normal agents' control inputs; i.e. $\vec{u}_{\N} = \bmx{u_{i_1}^T & \ldots & u_{i_{|\N|}}^T}^T$, $\{i_1,\ldots,i_{|\N|}\} \in \N$.

Our ultimate aim is to demonstrate that for all $t \geq 0$, 
\begin{align}
\label{eq:hdot}
    \dot{h}(\vec{x}(t)) + \alpha(h(\vec{x}(t))) \leq 0.
\end{align}
The dependence of $\vec{x}(t)$ on $t$ will be omitted for brevity.
Prior results typically focus on designing continuous $u(\cdot)$ functions that guarantee \eqref{eq:hdot} is satisfied. Satisfying \eqref{eq:hdot} in sampled-data systems for all intermediate times $t \in [t_k, t_{k+1})$, $k \in \Z_{\geq 0}$ is more challenging since $u(\cdot)$ is constant on each interval $t \in [k\Gamma, (k+1)\Gamma)$.
Inspired by \cite{cortez2019control}, this challenge will be addressed
as follows: given the sampled state $\vec{x}(t^k)$ and the state $\vec{x}(t)$, $t \in [t^k, t^{k+1})$, define the error term
\begin{align*}
    e(t,t^k) = \pth{\dot{h}(\vec{x}) - \dot{h}(\vec{x}^k)} + \pth{\alpha(h(\vec{x})) - \alpha(h(\vec{x}^k))}.
\end{align*}
From the LHS of \eqref{eq:hdot} we obtain
\begin{align*}
    &\dot{h}(\vec{x}) + \alpha(h(\vec{x})) = 
    \dot{h}(\vec{x}^k) + \pth{\dot{h}(\vec{x}) -\dot{h}(\vec{x}^k) } + \\ &\hspace{8em}\alpha(h(\vec{x}^k)) + \pth{\alpha(h(\vec{x})) - \alpha(h(\vec{x}^k))}, \\
    &\hspace{2em} = \dot{h}(\vec{x}^k) + \alpha(h(\vec{x}^k)) + e(t,t^k), \\
    & \hspace{2em}\leq \dot{h}(\vec{x}^k) + \alpha(h(\vec{x}^k)) + \sup_{t \in [t^k, t^{k+1})}\nrm{e(t,t^k)}.
\end{align*}
By defining a function $\eta(\cdot)$ such that $\eta(\Gamma) \geq \sup_{t \in [t^k, t^{k+1})} \nrm{e(t,t^k)}$, the inequality condition in \eqref{eq:hdot} is therefore satisfied for all times on the interval $t \in [t^k, t^{k+1})$ if for every $t^k \in \mathcal{T}$ the following condition holds:
\begin{align}
    &\dot{h}(\vec{x}^k) + \alpha(h(\vec{x}^k)) + \eta(\Gamma) \leq 0. \label{eq:hkdotstuff} 
\end{align}
Satisfaction of \eqref{eq:hkdotstuff} implies that $\dot{h}(\vec{x}) + \alpha(h(\vec{x})) \leq \dot{h}(\vec{x}^k) + \alpha(h(\vec{x}^k)) + \eta(\Gamma) \leq 0$ for all $t \in [t^k, t^{k+1})$.
To define such a function $\eta(\cdot)$, the following Lemma will be used.
\begin{lemma}
\label{lem:constants}
Consider the system \eqref{eq:generalsystem}. There exist constants $c_f, c_g, c_\alpha, c_\gamma, c_h \in \R$ such that for all $\vec{x}^1, \vec{x}^2 \in S$, all of the following inequalities hold:
\begin{align}
\sum_{i \in \N} \nrm{L_{f_i}h^{x_i}(\vec{x}^1) - L_{f_i}h^{x_i}(\vec{x}^2)} &\leq c_f \nrm{\vec{x}^1 - \vec{x}^2}, \label{eq:firstone}\\
    \sum_{i \in \N} \nrm{L_{g_i}h^{x_i}(\vec{x}^1) - L_{g_i}h^{x_i}(\vec{x}^2)} &\leq c_g \nrm{\vec{x}^1 - \vec{x}^2},\\
    \nrm{\alpha(h(\vec{x}^1)) - \alpha(h(\vec{x}^2))} &\leq c_{\alpha} \nrm{\vec{x}^1 - \vec{x}^2},\\
    \sum_{j \in \A} \nrm{\gamma_j^{\max}(\vec{x}^1) - \gamma_j^{\max}(\vec{x}^2)} &\leq c_{\gamma} \nrm{\vec{x}^1 - \vec{x}^2}, \label{eq:fourthone}\\
    \nrm{\sum_{l \in \V} L_{\phi_l} h^{x_l}(\vec{x}^1)} &\leq  c_h \sum_{l \in \V} \phi_l^{\max}\label{eq:fifthone} 
\end{align}
\end{lemma}

\begin{proof}
The inequalities \eqref{eq:firstone}-\eqref{eq:fourthone} follow from the fact that each $f_i$, $g_i$, $\alpha$, and $\gamma_j^{\max}$ are locally Lipschitz, $h \in C^{1,1}_{loc}$, and $S$ is compact. To demonstrate that \eqref{eq:fifthone} holds, observe that
$S$ being compact and $h \in C^{1,1}_{loc}$ implies $\nrm{\frac{\partial h(\vec{x})}{\partial x_l}}$ is bounded on $S$ for all $l \in \V$. Therefore, there exists a constant $c_h \in \R_{>0}$ such that for all $t \in [t^k, t^{k+1})$, $\forall \vec{x} \in S$,
\begin{align*}
    \nrm{\sum_{l \in \V} L_{\phi_l} h^{x_l}(\vec{x}^1)} \leq \sum_{l \in \V} \nrm{\frac{\partial h(\vec{x}^1)}{\partial x_l} \phi_l(t)} \leq c_h \sum_{l \in \V} \phi_l^{\max}
\end{align*}
\end{proof}

In addition to the inequalities in Lemma \ref{lem:constants}, observe that each set $\mathcal{U}_i$ being uniformly compact implies that there exist a constant $u_{\max} \geq 0$ such that $ \nrm{u_i^k} \leq u_{\max}$ for all $i \in \N$, $k \geq 0$. 
Using this definition of $u_{\max}$, the constants defined in Lemma \ref{lem:constants}, and the function $\eps(\cdot)$ in \eqref{eq:epsT} we define the function $\eta : \R_{\geq 0} \rarr \R$ as follows:
\begin{align}
\label{eq:mainetadef}
    \eta(\Gamma) = \pth{c_f + c_g u_{\max} + c_\alpha + c_{\gamma}} \eps(\Gamma) + c_h \sum_{l \in \V} \phi_l^{\max}.
\end{align}
The proof that $\eta(\Gamma) \geq \sup_{t \in [t^k, t^{k+1})} \nrm{e(t,t^k)}$ will be given in Theorem \ref{thm:multiagent}.
This definition of $\eta(\cdot)$ is used to define the following \emph{safety-preserving controls set} for the normal agents in $\N$:

\begin{align}
    K(\vec{x}) =& \big\{\vec{u}^k_{\N} \in \mathcal{U}_{\N} : \sum_{i \in \N} \bkt{L_{f_i} h^{x_i}(\vec{x}^k) + L_{g_i} h^{x_i}(\vec{x}^k) u_i^k} + \nonumber\\
     &\sum_{j \in \A} \gamma_j^{\max}(\vec{x}^k) + \alpha(h(\vec{x}^k)) + \eta(\Gamma) \leq 0 \big\} \label{eq:KNx}
\end{align}

Using this definition of $K(\cdot)$, the following Theorem presents conditions under which the set $S$ can be rendered forward invariant for the system \eqref{eq:generalsystem} with synchronous sampling times despite the actions of the adversarial agents.

\begin{theorem}
\label{thm:multiagent}
Consider the system \eqref{eq:generalsystem} with synchronous sampling times. If $\vec{x}^k \in S$ for $k \geq 0$, then for any control input $\vec{u}^k \in K(\vec{x}^k)$ the trajectory $\vec{x}(t)$ satisfies $\vec{x}(t) \in S$ for all $t \in [k\Gamma, (k+1)\Gamma)$.
\end{theorem}

\begin{proof}
First, denote $\dot{h}'(\vec{x}^k) = \dot{h}(\vec{x}^k) - \sum_{l \in \V} L_{\phi_l} h^{x_l}(\vec{x}^k)$. In words, $\dot{h}'(\vec{x}^k)$ is equal to $\dot{h}(\vec{x}^k)$ with all disturbance-related Lie derivatives subtracted out. Observe that
\begin{align*}
    &\dot{h}(\vec{x}^k) + \pth{\dot{h}(\vec{x}) -\dot{h}(\vec{x}^k) } = \dot{h}'(\vec{x}^k) + \sum_{l \in \V} L_{\phi_l}h^{x_l}(\vec{x}^k) + \\  &\hspace{2em} \pth{\dot{h}(\vec{x}) -  \dot{h}'(\vec{x}^k) - \sum_{l \in \V} L_{\phi_l}h^{x_l}(\vec{x}^k)},\\
    &=\dot{h}'(\vec{x}^k) + \pth{\dot{h}(\vec{x}) - \dot{h}'(\vec{x}^k)}, \\
    &= \dot{h}'(\vec{x}^k) + \pth{\dot{h}'(\vec{x}) - \dot{h}'(\vec{x}^k)} + \sum_{l \in \V} L_{\phi_l}h^{x_l}(\vec{x}).
\end{align*}
From \eqref{eq:generalsystem} and the definition of adversarial agents in \eqref{eq:adversarialu}, define the error term
\begin{align*}
    &e'(t,t^k) = \pth{\dot{h}'(\vec{x}) - \dot{h}'(\vec{x}^k)} + \sum_{l \in \V} L_{\phi_l}h^{x_l}(\vec{x}) + \\ & \hspace{6em} \pth{\alpha(h(\vec{x})) - \alpha(h(\vec{x}^k))},\\
    &= \pth{\sum_{i \in \N} L_{f_i}h^{x_i}(\vec{x}) - L_{f_i}h(\vec{x}^k)} + \sum_{l \in \V} L_{\phi_l}h^{x_l}(\vec{x}) + \\ 
    &\pth{\sum_{i \in \N} \bkt{L_{g_i}h^{x_i}(\vec{x}) - L_{g_i}h^{x_i}(\vec{x}^k)}u_i^k} + \\
    &\sum_{j \in \A} \pth{\gamma_j^{\max}(\vec{x}) - \gamma_j^{\max}(\vec{x}^k)} + \pth{\alpha(h(\vec{x})) - \alpha(h(\vec{x}^k))} 
\end{align*}
Since $t^{k+1} - t^{k} = (k+1)\Gamma - k\Gamma = \Gamma$ for all $k \geq 0$, by Lemma \ref{lem:xdiffbound} we have $\nrm{\vec{x} - \vec{x}^k} \leq \eps(\Gamma)$ for all $t \in [t^k, t^{k+1})$. 
Using Lemma \ref{lem:constants} and the definition of $\eta(\cdot)$ in \eqref{eq:mainetadef} yields the following upper bound on $\nrm{e'(t,t^k)}$:
\begin{align*}
    \sup_{t \in [t^k, t^{k+1})} \nrm{e'(t,t^k)} \leq& \pth{c_f + c_g u_{\max} + c_\alpha + c_{\gamma}}\eps(\Gamma) \\
    &+ c_h\sum_{l \in \V} \phi_l^{\max}, \\
    \implies \sup_{t \in [t^k, t^{k+1})} \nrm{e'(t,t^k)} \leq& \eta(\Gamma).
\end{align*}
Therefore for all $t \in [t^k, t^{k+1})$, it holds that
\begin{align*}
    &\dot{h}(\vec{x}) + \alpha(h(\vec{x})) = \dot{h}'(\vec{x}^k) + \pth{\dot{h}'(\vec{x}) - \dot{h}'(\vec{x}^k)} + \\
    &\hspace{2em}\sum_{l \in \V} L_{\phi_l}h^{x_l}(\vec{x}) +  \alpha(h(\vec{x}^k)) + \pth{\alpha(h(\vec{x})) - \alpha(h(\vec{x}^k))},\\
    &\hspace{2em} = \dot{h}'(\vec{x}^k) + \alpha(h(\vec{x}^k)) + e'(t,t^k),\\
    &\hspace{2em} \leq \dot{h}'(\vec{x}^k) + \alpha(h(\vec{x}^k)) + \sup_{t \in [t^k, t^{k+1})} \nrm{e'(t,t^k)},
\end{align*}
\begin{align*}
    &\hspace{2em} \leq \sum_{i \in \N} \bkt{L_{f_i} h^{x_i}(\vec{x}^k) + L_{g_i} h^{x_i}(\vec{x}^k) u_i^k} + \nonumber\\
     &\hspace{4em}\sum_{j \in \A} \gamma_j^{\max}(\vec{x}^k) + \alpha(h(\vec{x}^k)) + \eta(\Gamma).
\end{align*}
Choosing any $\vec{u}^k_{\N} \in K(\vec{x})$, observe from \eqref{eq:KNx} that
\begin{align}
    &\sum_{i \in \N} \bkt{L_{f_i} h^{x_i}(\vec{x}^k) + L_{g_i} h^{x_i}(\vec{x}^k) u_i^k} + \nonumber\\
     &\hspace{2em}\sum_{j \in \A} \gamma_j^{\max}(\vec{x}^k) + \alpha(h(\vec{x}^k) + \eta(\Gamma) \leq 0,\\
     &\implies \dot{h}(\vec{x}) + \alpha(h(\vec{x})) \leq 0.
\end{align}
Therefore any $\vec{u}_{\N}^k \in K(\vec{x}^k)$ renders the set $S$ forward invariant for all $t \in [t^k, t^{k+1})$. These arguments hold for all $k \in \Z_{\geq 0}$, which concludes the proof.
\end{proof}

\begin{remark}
Using Remark \ref{rem:gammaisdabomb} observe that given any $\vec{u}_{\N}^k \in K(\vec{x}^k)$, for all $u_j \in \mathcal{U}_j$, $j \in \A$ it holds that
\begin{align*}
    &\sum_{i \in \N}  \pth{L_{f_i}h^{x_i}(\vec{x}^k) + L_{g_i}h^{x_i}(\vec{x}^k)u_i} + \nonumber\\ 
    &\hspace{1em}\sum_{j \in \A} \pth{L_{f_j}h^{x_j}(\vec{x}^k) + L_{g_j}h^{x_j}(\vec{x}^k)u_j} + \alpha(h(\vec{x})) + \eta(\Gamma) \leq \\
    &\sum_{i \in \N}  \pth{L_{f_i}h^{x_i}(\vec{x}^k) + L_{g_i}h^{x_i}(\vec{x}^k)u_i} + \sum_{j \in \A} \gamma_j^{\max}(\vec{x}^k) + \\
    &\hspace{1em} \alpha(h(\vec{x})) + \eta(\Gamma).
\end{align*}
Therefore, the results of Theorem \ref{thm:multiagent} hold for any feasible control inputs $u_j \in \mathcal{U}_j$ for any agent $j \in \A$.

In other words, since the analysis of Theorem \ref{thm:multiagent} uses the maximum value functions $\gamma_j^{\max}(\cdot)$ for the contributions of the adversarial agents $j \in \A$ to the LHS of the safety condition \eqref{eq:Nagumosum}, the results of Theorem \ref{thm:multiagent} hold for any control inputs that the adversarial agents can apply within their respective feasible polytopes $\mathcal{U}_j$. In this sense the results of Theorem \ref{thm:multiagent} can be applied to a broader definition of the control inputs of agents in the set $\A$ than the one given in Section \ref{sec:problem}.
\end{remark}

When $K(\vec{x})$ defined in \eqref{eq:KNx} is nonempty, a feasible $\vec{u}_{\N}^* \in K(\vec{x})$ rendering $S$ invariant while minimally modifying $\vec{u}_{\textup{nom}}$ can be computed by solving the following QP:
\begin{alignat}{3}
		&\vec{u}_{\N}^*(\vec{x}^k) = & & \nonumber \\
		&\underset{\vec{u}_{\N} \in \mathcal{U}_{\N}}{\arg\min} & &  \nrm{\vec{u}_{\N} - \vec{u}_{\textup{nom}}}_2^2 \label{eq:QPmultiple} \\
		& \hspace{1em}\text{s.t.} & &  \sum_{i \in \N}  \pth{L_{f_i}h^{x_i}(\vec{x}^k) + L_{g_i}h^{x_i}(\vec{x}^k)u_i} + \nonumber\\ 
		& & & \hspace{2em} \sum_{j \in \A} \gamma_j^{\max}(\vec{x}^k) + \alpha(h(\vec{x})) + \eta(\Gamma) \leq 0\nonumber
\end{alignat}
Note that this QP requires the values of $\gamma_j^{\max}(\vec{x}^k)$, $j \in \A$, which can be solved for via a separate LP. 
Once $\vec{u}_{\N}^*(\vec{x}^k) \in K(\vec{x})$ has been obtained, each agent $i \in \N$ can then apply the local control input $u_i(\vec{x}^k)$.
By Theorem \ref{thm:multiagent}, safety of the entire system is guaranteed under the adversarial behavior for all forward time.
The case when $K(\vec{x})$ is empty is discussed in Section \ref{sec:maxpreserve}.

\subsection{Asynchronous Sampling Times}
\label{sec:asynchronous}

The assumption of identical, synchronous sampling times typically does not hold in practice.
In addition, a distributed system may not have access to a centralized entity to solve the QP in \eqref{eq:QPmultiple} to obtain $\vec{u}_{\N}$.
This subsection will therefore consider asynchronous sampling times and a distributed method for computing local control inputs.
Each agent $i \in \V$ is assumed to have a nominal sampling period $\Gamma_i \in \R_{>0}$ and the perturbed sequence of sampling times 
\begin{align}
\label{eq:perturbedT}
    \mathcal{T}_i = \{t_i^0, t_i^1, \ldots\} \text{ s.t. } t_i^{k+1} - t_i^{k} = \Gamma_i + \delta_i(k),\ \forall k \in \Z_{\geq 0},
\end{align}
where $\delta_i(k)$ is a disturbance satisfying $\nrm{\delta_i(k)} \leq \delta_i^{\max}$. 
The function $\delta_i$ can be used to model time delays due to disturbances such as clock asynchrony or packet drops in the communication network.
We denote $\Gamma^{\max} = \max_{i \in \V} \Gamma_i$ and $\delta^{\max} = \max_{i \in \V} \delta_i^{\max}$.
Recall from Section \ref{sec:problem} that we denote $\vec{x}^{k_i} = \vec{x}(t_i^k)$ and $u_i^{k_i} = u_i(t_i^k)$.

Each agent $i \in \N$ updates its control input $u_i^{k_i}$ at sampling times $t_i^k$ and also broadcasts $u_i^{k_i}$ to all other agents in the network. Each agent $i$ stores the values of the most recently received inputs from its normal in-neighbors $l \in \N$. The notation $\hat{u}_l^{k_i}$ denotes the most recently received input value by agent $i$ from agent $l$ at time $t_i^k$.

Using the definition of $\eta(\cdot)$ from \eqref{eq:mainetadef}, the following safety-preserving control set is defined for each $i \in \N$:
\begin{align*}
    \begin{aligned}
    &K_i(\vec{x}^{k_i}) = \Big\{u_i \in \mathcal{U}_i : L_{f_i} h^{x_i}(\vec{x}^{k_i}) + L_{g_i} h^{x_i}(\vec{x}^{k_i}) u_i \\
    &\hspace{1em} +\sum_{l \in \N \backslash \{i\}} \bkt{L_{f_l} h^{x_l}(\vec{x}^{k_i}) + L_{g_l} h^{x_l}(\vec{x}^{k_i}) \hat{u}_l^{k_i}} \\
     &\hspace{1em} + \sum_{j \in \A} \gamma_j^{\max}(\vec{x}^{k_i}) + \alpha(h(\vec{x}^{k_i}))  + \eta(\Gamma_i + \delta^{\max}) \leq 0 \Big\}
     \end{aligned}
\end{align*}
Theorem \ref{thm:async} presents conditions under which forward invariance of the set $S$ can be guaranteed for the distributed, asynchronous system described in this subsection.


\begin{theorem}
\label{thm:async}
Consider the system \eqref{eq:generalsystem} with sampling times described by \eqref{eq:perturbedT}. If at sampling time $t_i^k$ for $k \geq 0$, $i \in \N$ it holds that $\vec{x}^{k_i} \in S$, then for any $u_i^{k_i} \in K_i(\vec{x}^{k_i})$ the trajectory $\vec{x}(t)$ satisfies $\vec{x}(t) \in S$ for all $t \in [t_i^k, t_i^{k+1})$.
\end{theorem}

\begin{proof}
Choose any $i \in \N$ and consider the time interval $t \in [t_i^{k}, t_i^{k+1})$. Recall that $t_i^{k+1} - t_i^{k} \leq \Gamma_i + \delta^{\max}$ $\forall k \in \Z_{\geq 0}$ by virtue of \eqref{eq:perturbedT} and the definition of $\delta^{\max}$. In particular, this implies $\eps(\Gamma_i + \delta_i(k)) \leq \eps(\Gamma_i + \delta^{\max})$ for all $k \in \Z_{\geq 0}$ since $\eps(\cdot)$ is a class-$\mathcal{K}$ function in $\Gamma$.
For each $i \in \N$ define the value $e_i'(t,t^k)$ in a similar manner as Theorem \ref{thm:multiagent} and observe
\begin{align*}
    &\sup_{t \in [t_i^k, t_i^{k+1})} \nrm{e_i'(t,t_i^k)} \leq \\
    &\pth{c_f + c_g u_{\max} + c_\alpha + c_{\gamma}}\eps(\Gamma_i + \delta^{\max}) + c_h \sum_{l \in \V} \phi_l^{\max},\\
    &\implies \sup_{t \in [t_i^k, t_i^{k+1})} \nrm{e_i'(t,t_i^k)} \leq \eta(\Gamma_i + \delta^{\max})
\end{align*}
The same logic as in Theorem \ref{thm:multiagent} can then be used to demonstrate that $\dot{h}(\vec{x}) + \alpha(h(\vec{x})) \leq 0$ for all $t \in [t_i^k, t_i^{k+1})$.
\end{proof}
Under the communication protocol described previously, each agent can use the most recently received inputs $\hat{u}_l^{k_i}$ from other normal agents to calculate a control input $u_i^{k_i} \in K_i(\vec{x}^{k_i})$. Such a $u_i^{k_i}$ can be computed by solving the following QP:
\begin{alignat}{3}
		&u_i(\vec{x}^{k_i}) = 
		&&\underset{u_i \in \mathcal{U}_i}{\arg\min} \nrm{u_i - u_{i,\textup{nom}}^{k_i}}_2^2 \label{eq:QPdistributed} \\
		& \hspace{1em}\text{s.t.} & &   \pth{L_{f_i}h^{x_i}(\vec{x}^{k_i}) + L_{g_i}h^{x_i}(\vec{x}^{k_i})u_i} + \nonumber\\
		& & & \sum_{l \in \N \backslash \{i\}} \pth{L_{f_l}h^{x_l}(\vec{x}^{k_i}) + L_{g_l}h^{x_l}(\vec{x}^{k_i})\hat{u}_l^{k_i}} + \nonumber\\
		& & & \hspace{2em} \sum_{j \in \A} \gamma_j^{\max}(\vec{x}^{k_i}) + \alpha(h(\vec{x}^{k_i})) + \nonumber\\
		& & & \hspace{2em}\eta(\Gamma_i + \delta^{\max}) \leq 0. \nonumber
\end{alignat}
Like the previous formulations, the values of $\gamma_j^{\max}(\cdot)$ for $j \in \A$ can be calculated via solving a separate LP.
By the results of Theorem \ref{thm:async}, when each $K_i(\vec{x})$ is nonempty and each normal agent applies the controller defined by \eqref{eq:QPdistributed} the multi-agent safe set is rendered forward invariant despite any collective worst-case behavior of the adversarial agents. 

\subsection{Maximum Safety-Preserving Control Action}
\label{sec:maxpreserve}

One of the required conditions of the foregoing results is the nonemptiness of the safety-preserving controls sets $K(\vec{x})$ and $K_i(\vec{x})$, which is also closely related to the feasibility of the respective QPs \eqref{eq:QPmultiple}, \eqref{eq:QPdistributed}. Conditions under which such sets remain nonempty for general systems remains an open question.
Guaranteeing both safety and the feasibility of the QP calculating the control input $u_i(\vec{x}^{k_i})$ has been a recent topic of study \cite{garg2019prescribed, black2020quadratic}, and can depend on the choice of extended class-$\mathcal{K}_\infty$ function $\alpha(\cdot)$.

In contrast, consider the sampled-data control law $u_i^{\min}(\cdot)$ defined in \eqref{eq:umin}.
Intuitively speaking, \eqref{eq:umin} represents the strongest control effort agent $i \in \N$ can apply towards minimizing the LHS of \eqref{eq:Nagumosum}. This control input can be solved for by taking the $\arg\min$ of the minimizing LP in \eqref{eq:LPvminmax}:
\begin{alignat}{3}
\begin{aligned}
\label{eq:LPnormal}
		u_i^{\min}(\vec{x}^{k_i}) =\ 
		&\underset{u_i \in \R^{m_i}}{\arg\min} & & L_{g_i}h^{x_i}(\vec{x}^{k_i}) u_i \\
		& \hspace{1em}\text{s.t.} & &   A_i(\vec{x}^{k_i}) u_i \leq b_i(\vec{x}^{k_i})
\end{aligned}
\end{alignat}
For any system satisfying Assumption 3, the set $\mathcal{U}_i(\vec{x}) = \{u : A_i(\vec{x}) u \leq b_i(\vec{x})\}$ is nonempty for all $\vec{x} \in S$. 
This implies that \eqref{eq:LPnormal} is always guaranteed to be feasible for $\vec{x} \in S$.
However the question remains as to when the control action \eqref{eq:umin} can guarantee forward invariance of $S$.
Towards this end, define the set
\begin{align}
    \partial S_\eps = \brc{x \in S : \min_{z \in \partial S} \nrm{x - z} \leq \eps },\quad \eps > 0.
    \label{eq:Seps}
\end{align}
In words, $\partial S_\eps$ is an ``inner boundary region" of $S$ which includes all points in $S$ within distance $\eps$ of $\partial S$ with respect to a chosen norm.
The next theorem presents a sufficient condition for when each normal agent applying $u_i^{\min}(\cdot)$ renders $S$ forward invariant in the presence of an adversarial set $\A$.

\begin{theorem}
\label{thm:asyncnormalmaxeffort}
Let $\eps^* = \eps(\Gamma^{\max} + 2 \delta^{\max})$ and define the sets $\partial S_{\eps^*}, \partial S_{2\eps^*}$ as per \eqref{eq:Seps}.
Suppose that each normal agent $i \in \N$ applies the control input $u_i^{\min}(\vec{x}^{k_i})$ from \eqref{eq:umin} for all sampled states $\vec{x}^{k_i}$ satisfying $\vec{x}^{k_i} \in \partial S_{2\eps^*}$.
Then $S$ is forward invariant if $\vec{x}(0) \in S \backslash \partial S_{2\eps^*}$ and the following condition holds:
\begin{align}
    \max_{\vec{x} \in \partial S_{2\eps^*}}\Bigg[
    &\sum_{i \in \N} \max_{\vec{x}^i \in B(\vec{x}, \eps^*)} \bkt{\gamma_i^{\min}(\vec{x}^i)} +  \label{eq:asyncminmaxcondition}\\
    &\sum_{j \in \A} \gamma_j^{\max}(\vec{x}) +
    \alpha(h(\vec{x})) \Bigg] \leq - \eta(\Gamma^{\max} + 2\delta^{\max}). \nonumber
\end{align}
\end{theorem}

\begin{proof}
The proof first demonstrates that the most recently sampled states of all agents always lie within a closed ball of radius $\eps^* = \eps(\Gamma^{\max} + 2 \delta^{\max})$. 
Next, it shows that $\vec{x}(0) \in S \backslash \partial S_{\eps^*}$ implies that $\vec{x}(t)$ cannot leave $S$ without all agents sampling the state at least once within the region $\partial S_{\eps^*}$.
Finally, it is shown that this fact combined with \eqref{eq:asyncminmaxcondition} implies that $S$ is forward invariant.

Choose any $i \in \N$ and any sampling time $t_i^k$ for agent $i$. By the definition of $\Gamma^{\max}$ and $\delta^{\max}$, the next sampling time $t_i^{k+1}$ satisfies $t_i^{k+1} - t_i^{k} \leq t_i^{k} + \Gamma^{\max} + 2 \delta^{\max}$. Since this holds for all $i \in \V$, given any $i_1, i_2 \in \N$ and interval $[t_{i_1}^k, t_{i_1}^k + \Gamma^{\max} + \delta^{\max}]$, there exists a sampling time for $i_2$ satisfying $t_{i_2}^{k'} \in [t_{i_1}^k, t_{i_1}^k + \Gamma^{\max} + 2\delta^{\max}]$. Using Lemma \ref{lem:xdiffbound}, this implies that the maximum normed difference between any two most recently sampled states $\vec{x}(t_{i_1}^{k^*})$ and $\vec{x}(t_{i_2}^{k^*})$ satisfies $\nrm{\vec{x}(t_{i_1}^{k^*}) - \vec{x}(t_{i_2}^{k^*})} \leq \eps(\Gamma^{\max} + 2\delta^{\max}) = \eps^*$. Since this holds for all $i_1, i_2 \in \V$ at any $t_{i_1}^k$, the most recently sampled states of all agents therefore always lie within a ball of radius $\eps^*$.

Next, consider any agent $i$ with sampling time $t_i^{k}$ such that $\vec{x}(t_i^k) \in S \backslash \partial S_{\eps^*}$ and $\vec{x}(t_i^{k+1}) \not\in S \backslash \partial S_{\eps^*}$. Since $\nrm{\vec{x}(t_i^{k+1}) - \vec{x}(t_i^k)} \leq \eps^*$ by previous arguments, this implies that $\vec{x}(t_i^{k+1}) \in \partial S_{\eps^*}$. 
Therefore $\vec{x}(0) \in S \backslash \partial S_{\eps^*}$ implies that $\vec{x}$ cannot leave $S$ without each agent $i \in \N$ having at least one sampling time $t_i^{k}$ such that $\vec{x}^{k_i} \in \partial S_{\eps^*}$.
Note that $\vec{x}(0) \in S \backslash \partial S_{2\eps^*}$ as per the Theorem statement implies that $\vec{x}(0) \in S \backslash \partial S_{\eps^*}$ since $\partial S_{\eps^*} \subset \partial S_{2\eps^*}$.

Define $e'(t,t_i^k)$ in a similar manner to Theorem \ref{thm:multiagent}. Observe that $t_i^{k+1} - t_i^k \leq \Gamma^{\max} + 2\delta^{\max}$ for all $i \in \N$. In addition, for any $i_1, i_2 \in \N$ with most recent sampling times $t_{i_1}^{k_{i_1}}$ and $t_{i_2}^{k_{i_2}}$, it can be shown that $|t_{i_1}^{k_{i_1}} - t_{i_2}^{k_{i_2}}| \leq \Gamma^{\max} + 2\delta^{\max}$.
Therefore on any interval $t \in [t_{i_1}^{k_{i_1}}, t_{i_2}^{k_{i_2}})$, we have
\vspace{-0.1em}
\begin{align*}
    &\sup_{t \in [t_{i_1}^{k_{i_1}}, t_{i_2}^{k_{i_2}})} \nrm{e'(t,t_{i_1}^{k_{i_1}})} \leq \\
    &\pth{c_f + c_g u_{\max} + c_\alpha + c_{\gamma}}\eps^* + c_h \sum_{l \in \V} \phi_l^{\max},\\
    &\implies \sup_{t \in [t_{i_1}^{k_{i_1}}, t_{i_2}^{k_{i_2}})} \nrm{e'(t,t_{i_1}^{k_{i_1}})} \leq \eta(\Gamma^{\max} + 2\delta^{\max}).
\end{align*}
Choose the first sampling time $t_{i_1}^{k_{i_1}}$ such that $t_{i_1}^{k_{i_1}} \geq \Gamma^{\max} + 2\delta^{\max}$ and $\vec{x}^{k_{i_1}} \in \partial S_{\eps^*} \subset \partial S_{2\eps^*}$.
Since $\vec{x}(0) \in S_{2\eps^*}$ by the Theorem statement, it can be shown using prior arguments that such a sampling time is guaranteed to exist.
This choice of $t_{i_1}^{k_{i_1}}$ implies that all agents have sampled at least once at or before $t_{i_1}^{k_{i_1}}$.
Let $t_{i_2}^{k_{i_2}} > t_{i_1}^{k_{i_1}}$ be the next normal agent sampling time strictly greater than $t_{i_1}^{k_{i_1}}$, with the associated agent denoted $i_2 \in \N$. Let $\vec{x}^{k_{i_1}},\ldots, \vec{x}^{k_{i_{|\N|}}}$ denote the most recently sampled states of all normal agents. By prior arguments $\vec{x}^{k_{i_l}} \in \overbar{B}(\vec{x}^{k_{i_1}}, \eps^*)$ for all $l \in 1,\ldots,|\N|$, and therefore by \eqref{eq:asyncminmaxcondition} at time $t_{i_1}^{k_{i_1}}$ we have
\begin{align*}
    &\sum_{p \in 1,\ldots,|\N|} \gamma_i^{\min}(\vec{x}^{k_{i_p}}) - \sum_{j \in \A} \gamma_j^{\max}(\vec{x}^{k_{i_1}}) + \alpha(h(\vec{x}^{k_{i_1}})) + \\
    &\hspace{2em} \eta^*(\Gamma^{\max} + 2\delta^{\max}) \leq 0.
\end{align*}
From this it holds that for all $t \in [t_{i_1}^{k_{i_1}}, t_{i_2}^{k_{i_2}})$ we have
\begin{align*}
    &\dot{h}(\vec{x}) + \alpha(h(\vec{x})) \leq \sum_{p \in 1,\ldots,N} \gamma_i^{\min}(\vec{x}^{k_{i_p}}) - \sum_{j \in \A} \gamma_j^{\max}(\vec{x}^{k_{i_1}}) + \\ &\hspace{2em} \alpha(h(\vec{x}^{k_{i_1}})) +
     \eta^*(\Gamma^{\max} + 2\delta^{\max}) \leq 0. 
\end{align*}
It follows that $S$ is forward invariant on the interval $t \in [t_{i_1}^{k_{i_1}}, t_{i_2}^{k_{i_2}})$. 
The preceding arguments can be repeated for any subsequent adjacent sampling times $t_{i_l}^{k_{i_l}}, t_{i_p}^{k_{i_p}}$, $t_{i_l}^{k_{i_l}} < t_{i_p}^{k_{i_p}}$ to show that $S$ is forward invariant on $[t_{i_l}^{k_{i_l}}, t_{i_p}^{k_{i_p}})$, which concludes the proof.
\end{proof}

\section{Safe Set Functions with High Relative Degree}
\label{sec:highdegree}

 It has been demonstrated in prior literature that there exist safe set functions $h$ where agents' control inputs do not appear in the expression for the time derivative $\dot{h}(\vec{x})$, i.e., $\frac{\partial h(\vec{x})}{\partial x_i} g_i(\vec{x}) = \bm 0$ for all $\vec{x}$ \cite{xiao2019control, nguyen2016exponential}.
 These functions are said to have \emph{high relative degree with respect to the system dynamics}.
 In such cases, prior literature has considered methods for computing continuous-time controllers which provably maintain forward invariance of the safe set. These prior results do not consider systems with sampled-data dynamics however, nor do they consider the presence of agents behaving in an adversarial manner. In this section we extend our previous results to consider a class of safe set functions having high relative degree w.r.t. system dynamics.

 In prior work, safety of systems without disturbances and having continuous control inputs using safe set functions having high relative degree w.r.t. system dynamics is typically considered as follows: a function $h : \R^{\bar{n}} \rarr \R$ describing the safe set is used to define a series of functions $\psi_j : \R^{\bar{n}} \rarr \R$, $j=1,\ldots,q$ in the following manner:
 \begin{align}
\begin{aligned}
\label{eq:oldpsifuncs}
    \psi_0(\vec{x}) &\triangleq h(\vec{x}), \\
    \psi_1(\vec{x}) &\triangleq \dot{\psi}_0(\vec{x}) + \alpha_1(\psi_0(\vec{x})), \\
    &\vdots \\
    \psi_q(\vec{x}) &\triangleq \dot{\psi}_{q-1} + \alpha_{q}(\psi_{q-1}(\vec{x})),
\end{aligned}
\end{align}
where each $\alpha_j : \R \rarr \R$ is an extended class-$\mathcal{K}_\infty$ function. The integer $q \in \Z_{\geq 1}$ is chosen to be the smallest integer such that a control input $u_i$ for some $i \in \V$ appears in the expression for $\psi_q(\vec{x})$. The integer $q$ is called the \emph{relative degree} of $h$ w.r.t the system dynamics.
The functions in \eqref{eq:oldpsifuncs} are associated with the following series of sets:
\begin{align}
\begin{aligned}
\label{eq:Spsisets}
    S_1 &\triangleq \{\vec{x} \in \R^{\bar{n}} : \psi_0(\vec{x}) \leq 0 \}. \\
    S_2 &\triangleq \{\vec{x} \in \R^{\bar{n}} : \psi_1(\vec{x}) \leq 0 \}. \\
    \vdots \\
    S_q &\triangleq \{\vec{x} \in \R^{\bar{n}} : \psi_{q-1}(\vec{x}) \leq 0 \}. \\
    \end{aligned}
\end{align}
For brevity, we denote $S_I \triangleq \bigcap_{r=1}^p S_r$.
The following result from prior literature applies to systems with \emph{continuous} control inputs:
\begin{theorem}[\cite{xiao2019control}]
\label{thm:continuousHOCBF}
    Suppose $\vec{x}(t_0) \in \bigcap_{i=1}^p S_i$. Then the set $\bigcap_{i=1}^q S_i$ is rendered forward invariant under any Lipschitz continuous controller $\vec{u}(t)$ that ensures the condition $\psi_q(\vec{x}(t)) \leq 0$ for all $t \geq t_0$.
\end{theorem}
However, this prior result considers continuous control inputs, does not account for the disturbances $\psi_i(t)$ in \eqref{eq:generalsystem}, and does not consider the presence of agents behaving in an adversarial manner.

This section will extend the results in the previous section to present a method for normally-behaving agents with the sampled-data dynamics \eqref{eq:generalsystem} to maintain safety using a safe set function $h$ with high relative degree w.r.t. \eqref{eq:generalsystem} in the presence of adversarial agents.
First, to address the presence of the disturbances $\phi_i(t)$, $i \in \V$, recall from Lemma \ref{lem:constants} that there exists a constant $c_h \geq 0$ such that $\nrm{\sum_{i \in \V} L_{\phi_i} h^{x_i}(\vec{x})} \leq c_h \sum_{i \in \V} \phi_i^{\max}$. We define the constant
\begin{align}
\label{eq:xi}
    \xi  = c_h \sum_{i \in \V} \phi_i^{\max}.
\end{align}
The function $h(\vec{x})$ and constant $\xi$ are used to define a series of functions $\psi_j^d : \R^{\bar{n}} \rarr \R$, $j=1,\ldots,q$ in the following manner:
\begin{align}
\begin{aligned}
\label{eq:psifuncs}
    \psi_0^d(\vec{x}) &\triangleq h(\vec{x}), \\
    \psi_1^d(\vec{x}) &\triangleq \sum_{i \in \V}L_{f_i} h^{x_i}(\vec{x}) + \xi + \alpha_1(\psi_0^d(\vec{x})), \\
    \psi_2^d(\vec{x}) &\triangleq \dot{\psi}_1^d(\vec{x}) + \alpha_2(\psi_1^d(\vec{x})), \\
    &\vdots \\
    \psi_q^d(\vec{x}) &\triangleq \dot{\psi}_{q-1}^d + \alpha_{q}(\psi_{q-1}^d(\vec{x})),
\end{aligned}
\end{align}
where each $\alpha_j : \R \rarr \R$ is an extended class $\mathcal{K}_\infty$ function and is locally Lipschitz on $\R$.
We make the following assumptions:
\begin{assume}
\label{assume:relativedeg}
The agent inputs $u_i$ for all $i \in \V$ appear simultaneously in $\psi_q^d(\vec{x})$, $q \in \Z_{\geq 1}$, and are all absent in all $\psi_{j}^d$, $0 \leq j < q$. 
\end{assume}
\begin{assume}
\label{assume:psiLC}
The function $\psi_{q-1}^d$ satisfies $\psi_{q-1}^d \in \mathcal{C}_{loc}^{1,1}$.
\end{assume}
In particular, this section considers cases where the relative degree $q > 1$, since cases where $q = 1$ can be treated by the results in Section \ref{sec:mainresults}.
The sets $S_1^d, \ldots, S_q^d$ and $S_I^d$ are defined as
\begin{align}
\begin{aligned}
\label{eq:Spsisets}
    S_1^d &\triangleq \{\vec{x} \in \R^{\bar{n}} : \psi_0^d(\vec{x}) \leq 0 \}. \\
    S_2^d &\triangleq \{\vec{x} \in \R^{\bar{n}} : \psi_1^d(\vec{x}) \leq 0 \}. \\
    \vdots \\
    S_q^d &\triangleq \{\vec{x} \in \R^{\bar{n}} : \psi_{q-1}^d(\vec{x}) \leq 0 \} \\
    S_I^d &\triangleq \bigcap_{k=1}^q S_k^d.
    \end{aligned}
\end{align}

The following Lemma will be needed for our main result.
It allows the analysis to consider disturbances $\phi_i(t)$ which are not differentiable in time.
\begin{lemma}
\label{lem:getridofdist}
Let $h$ have relative degree $q > 1$ with respect to \eqref{eq:generalsystem}. Then it holds that $\dot{\psi}_0^d(\vec{x}) + \alpha_1(\psi_0^d(\vec{x})) \leq \psi_1^d(\vec{x})$ $\forall t \geq 0$.
\end{lemma}

\begin{proof}
Since $q > 0$, by Assumption \ref{assume:relativedeg} the time derivative of $\psi_0^d(\vec{x})$ satisfies $\dot{\psi}_0^d(\vec{x}) = \sum_{i \in \V} L_{f_i}h^{x_i}(\vec{x}) +  L_{\phi_i} h^{x_i}(\vec{x})$.
From Lemma \ref{lem:constants} and equation \eqref{eq:xi} we have $\nrm{\sum_{i \in \V} L_{\phi_i} h^{x_i}(\vec{x})} \leq c_h \sum_{i \in \V} \phi_i^{\max} = \xi$. Using \eqref{eq:psifuncs} it follows that
\begin{align*}
    \dot{\psi}_0^d(\vec{x}) + \alpha_1(\psi_0^d(\vec{x})) &\leq \sum_{i \in \V} \pth{L_{f_i}h^{x_i}(\vec{x})} + \xi + \alpha_1(\psi_0^d(\vec{x})), \\
    &= \psi_1^d(\vec{x}),
\end{align*}
which concludes the proof.
\end{proof}
By upper bounding the term $L_{\phi_i} h^{x_i}(\vec{x})$ with the constant $\xi$, no time derivatives of $\phi_i(t)$ appear in the functions $\psi_2,...\psi_q$.

Similar to Theorem \ref{thm:continuousHOCBF}, to achieve forward invariance of $S_I^d$ under a ZOH control law the key condition is to show that $\psi_q^d(\vec{x}(t), \vec{u}(t)) \leq 0$ for all $t \geq t_0$. Using a similar method as the prior section, for a ZOH $\vec{u}^k$ we can define the error term
\begin{align}
\label{eq:errorpsi}
    e_{\psi}(t,t^k) &= \pth{\psi_q^d(\vec{x}) - \psi_q^d(\vec{x}^k)}.
\end{align}
For all $t \in [t_i^k, t_i^{k+1})$ it therefore holds that
\begin{align*}
    \psi_q^d(\vec{x}(t)) &= \psi_q^d(\vec{x}^k) + e_{\psi}(t,t_i^k) \\
    &\leq \psi_q^d(\vec{x}^k) + \sup_{t \in [t_i^k, t_i^{k+1})} \nrm{e_{\psi}(t,t^k)}.
\end{align*}
If it holds that $\psi_q^d(\vec{x}^k) + \sup_{t \in [t_i^k, t_i^{k+1})} \nrm{e_{\psi}(t,t^k)} \leq 0$, then for all $t \in [t_i^k, t_i^{k+1})$ we therefore have
$\psi_q^d(\vec{x}) \leq 0$ for all $t \in [t_i^k, t_i^{k+1})$. 

Consider the asynchronous system with perturbed sampling times from section \ref{sec:asynchronous} such that Assumption \ref{assume:relativedeg} is satisfied and the function $h$ has relative degree $q$ under \eqref{eq:generalsystem}.
Using \eqref{eq:generalsystem} and \eqref{eq:psifuncs}, the function $\psi_q(\vec{x})$ can be expanded into the expression
\begin{align}
    \psi_q(\vec{x}) =& \dot{\psi}_{q-1}(\vec{x}) + \alpha_q(\psi_{q-1}(\vec{x})), \nonumber\\
    =& \sum_{i \in \N} L_{f_i} \psi_{q-1}^{x_i}(\vec{x}) + L_{g_i} \psi_{q-1}^{x_i}(\vec{x}) u_i + \label{eq:psi_q_extended}\\
    &\sum_{j \in \N} L_{f_j} \psi_{q-1}^{x_j}(\vec{x}) + L_{g_j} \psi_{q-1}^{x_j}(\vec{x}) u_j + \alpha_q(\psi_{q-1}(\vec{x})) \nonumber 
\end{align}
Observe that the RHS of \eqref{eq:psi_q_extended} is affine in $\vec{u}$. This follows from \eqref{eq:psi_q_extended} and the definition of the relative degree $q$ from Assumption \ref{assume:relativedeg}.
Similar to equation \eqref{eq:vLPs}, define the functions
\begin{align}
\begin{aligned}
    \hat{\gamma}_i^{\min}(\vec{x}) &= \min_{u_i \in \mathcal{U}_i} \bkt{L_{f_i}\psi_{q-1}^{x_i}(\vec{x}) + L_{g_i}\psi_{q-1}^{x_i}(\vec{x}) u_i}, \\
    \hat{\gamma}_i^{\max}(\vec{x}) &= \max_{u_i \in \mathcal{U}_i} \bkt{L_{f_i}\psi_{q-1}^{x_i}(\vec{x}) + L_{g_i}\psi_{q-1}^{x_i}(\vec{x}) u_i}. \label{eq:vLPpsi}
    \end{aligned}
\end{align}
As in the previous section, the functions $\hat{\gamma}_i^{\min}$, $\hat{\gamma}_i^{\max}$ can be shown to be locally Lipschitz on the set $S_I$.
\begin{lemma}
\label{lem:vLipsPsi}
If the interior of $\mathcal{U}_i(\vec{x})$ is nonempty for all $\vec{x} \in S_I$ and $\mathcal{U}_i(\vec{x})$ is uniformly compact near $\vec{x}$ for all $\vec{x} \in S_I$, then the functions $\hat{\gamma}_i^{\min}(\cdot)$ and $\hat{\gamma}_i^{\max}(\cdot)$ defined by \eqref{eq:vLPpsi} are locally Lipschitz on $S_I$.
\end{lemma}

\begin{proof}
The result follows from Assumption \ref{assume:psiLC} and by using similar arguments as in Lemma \ref{lem:vLips}.
\end{proof}

Similar to Section \ref{sec:mainresults}, the following result will be needed to define a function $\eta': \R_{\geq 0} \rarr \R$ that will be used to upper bound the normed error term $\sup_{t \in [t_i^k, t_i^{k+1})} \nrm{e_{\psi}(t,t^k)}$:

\begin{lemma}
\label{lem:psiconstants}
Consider the system \eqref{eq:generalsystem} and the function $\psi_{q-1}^d(\vec{x})$. There exist constants $c_f', c_g', c_\alpha', c_{\hat{\gamma}}' \in \R$ such that for all $\vec{x}^1, \vec{x}^2 \in S_I^d$, all of the following inequalities hold:
{
\small
\begin{align*}
\sum_{i \in \N} \nrm{L_{f_i}(\psi_{q-1}^d)^{x_i}(\vec{x}^1) - L_{f_i}(\psi_{q-1}^d)^{x_i}(\vec{x}^2)} &\leq c_f' \nrm{\vec{x}^1 - \vec{x}^2},\\
    \sum_{i \in \N} \nrm{L_{g_i}(\psi_{q-1}^d)^{x_i}(\vec{x}^1) - L_{g_i}(\psi_{q-1}^d)^{x_i}(\vec{x}^2)} &\leq c_g' \nrm{\vec{x}^1 - \vec{x}^2},\\
    \nrm{\alpha_q(\psi_{q-1}^d(\vec{x}^1)) - \alpha_q(\psi_{q-1}^d(\vec{x}^2))} &\leq c_{\alpha}' \nrm{\vec{x}^1 - \vec{x}^2},\\
    \sum_{j \in \A} \nrm{\hat{\gamma_j}^{\max}(\vec{x}^1) - \hat{\gamma_j}^{\max}(\vec{x}^2)} &\leq c_{\hat{\gamma}}' \nrm{\vec{x}^1 - \vec{x}^2},
\end{align*}
}
\end{lemma}

\begin{proof}
Follows from $\psi_{q-1} \in \mathcal{C}_{loc}^{1,1}$ by Assumption \ref{assume:psiLC}, from $\alpha_q$ being locally Lipschitz on $\R$ by definition, and from $\hat{\gamma}_i^{\min}, \hat{\gamma}_i^{\max}$ being locally Lipschitz by Lemma \ref{lem:vLipsPsi}.
\end{proof}

Using the constants defined in Lemma \ref{lem:psiconstants} and the function $\eps(\cdot)$ in \eqref{eq:epsT}, we define the function $\eta' : \R_{\geq 0} \rarr \R$ as follows:
\begin{align}
\label{eq:etaprimedef}
    \eta'(\Gamma) = \pth{c_f' + c_g' u_{\max} + c_\alpha' + c_{\hat{\gamma}}'} \eps(\Gamma).
\end{align}
This definition of $\eta'(\cdot)$ is used to define the following safety-preserving controls sets for $i \in \V$. Recall from Section \ref{sec:asynchronous} that $\hat{u}_l^{k_i}$ denotes the most recently received input value by agent $i \in \N$ from agent $l \in \N$ at time $t_i^k$.
\begin{align*}
    &K_i^{\psi}(\vec{x}^{k_i}) = \brc{u_i \in \mathcal{U}_i : \psi_q(\vec{x}^{k_i}) \leq 0},\\
    &\hspace{1em} = \Big\{u_i \in \mathcal{U}_i : L_{f_i} \psi_{q-1}^{x_i}(\vec{x}^{k_i}) + L_{g_i} \psi_{q-1}^{x_i}(\vec{x}^{k_i}) u_i \\
    &\hspace{1em} +\sum_{l \in \N \backslash \{i\}} \bkt{L_{f_l} \psi_{q-1}^{x_l}(\vec{x}^{k_i}) + L_{g_l} \psi_{q-1}^{x_l}(\vec{x}^{k_i}) \hat{u}_l^{k_i}} \\
     &\hspace{1em} + \sum_{j \in \A} \gamma_j^{\max}(\vec{x}^{k_i}) + \alpha(\psi_{q-1}(\vec{x}^{k_i}))  + \eta'(\Gamma_i + \delta^{\max}) \leq 0 \Big\}.
\end{align*}
The next Theorem demonstrates conditions under which the set $S$ may be rendered forward invariant for trajectories of the system \eqref{eq:generalsystem}.

\begin{theorem}
\label{thm:ARHO_CBF}
Consider the system \eqref{eq:generalsystem} with sampling times described by \eqref{eq:perturbedT}.
Let $\psi_1^d,\ldots,\psi_q^d$ be defined as in \eqref{eq:psifuncs}.
If at sampling time $t_i^k$ for $k \geq 0$, $i \in \N$ it holds that $\vec{x}^{k_i} \in S_I$, then for any $u_i^{k_i} \in K_i^{\psi}(\vec{x}^{k_i})$ the trajectory $\vec{x}(t)$ satisfies $\vec{x}(t) \in S_I$ for all $t \in [t_i^k, t_i^{k+1})$.
\end{theorem}

\begin{proof}
From \eqref{eq:errorpsi} and \eqref{eq:psifuncs}, we have
\begin{align}
    e_{\psi}(t,t^k) =& \pth{\psi_q^d(\vec{x}) - \psi_q^d(\vec{x}^k)}, \nonumber \\
    =& \sum_{i \in \N} \pth{L_{f_i} (\psi_{q-1}^d)^{x_i}(\vec{x}) - L_{f_i} (\psi_{q-1}^d)^{x_i}(\vec{x}^k)} + \nonumber\\
    &\sum_{i \in \N} \pth{L_{g_i}(\psi_{q-1}^d)^{x_i}(\vec{x}) - L_{g_i}(\psi_{q-1}^d)^{x_i}(\vec{x}^k)}u_i^k + \nonumber\\
    & \sum_{j \in \A} \pth{\hat{\gamma}_j^{\max}(\vec{x}) - \hat{\gamma}_j^{\max}(\vec{x}^k)} + \nonumber\\
    &\pth{\alpha_q(\psi_{q-1}(\vec{x})) - \alpha_q(\psi_{q-1}(\vec{x}^k))} \label{eq:idontcareanymore}
\end{align}
Choose any $i \in \N$ and consider the time interval $t \in [t_i^{k}, t_i^{k+1})$. Recall that $t_i^{k+1} - t_i^{k} \leq \Gamma_i + \delta^{\max}$ $\forall k \in \Z_{\geq 0}$ by virtue of \eqref{eq:perturbedT} and the definition of $\delta^{\max}$. In particular, this implies $\eps(\Gamma_i + \delta_i(k)) \leq \eps(\Gamma_i + \delta^{\max})$ for all $k \in \Z_{\geq 0}$ since $\eps(\cdot)$ is a class-$\mathcal{K}$ function in $\Gamma$. Using equations \eqref{eq:idontcareanymore}, \eqref{eq:etaprimedef}, Lemma \ref{lem:psiconstants}, and Lemma \eqref{lem:xdiffbound} observe that 
\begin{align*}
    &\sup_{t \in [t_i^{k}, t_i^{k+1})} \nrm{e_{\psi}(t,t_i^k)} \leq \\
    &\pth{c_f' + c_g' u_{\max} + c_{\alpha}' + c_{\hat{\gamma}}'}\eps(\Gamma_i + \delta^{\max}), \\
    &\implies \sup_{t \in [t_i^{k}, t_i^{k+1})} \nrm{e_{\psi}(t,t_i^k)} \leq \eta'(\Gamma_i + \delta^{\max})
\end{align*}
The same logic as in Theorem \ref{thm:multiagent} can then be used to demonstrate that for any $u_i \in K_i^{\psi}(\vec{x}^{k_i})$ it holds that $\psi_q^d(\vec{x}(t)) \leq \psi_q^d(\vec{x}^k) + \eta'(\Gamma_i + \delta^{\max}) \leq 0$ for all $t \in [t_i^k, t_i^{k+1})$.

We next demonstrate that $\psi_q^d(\vec{x}) \leq 0$ for all $t \in [t_i^k, t_i^{k+1})$ implies that $\vec{x} \in S_I^d$ $\forall t \in [t_i^k, t_i^{k+1})$.
For brevity, denote $I_i^k = [t_i^k, t_i^{k+1})$.
Since $\psi_q^d(\vec{x}) \leq 0$ for all $t \in I_i^k$, from \eqref{eq:psifuncs} this implies that $\dot{\psi}_{q-1}^d(\vec{x}) + \alpha_q(\psi_{q-1}^d(\vec{x})) \leq 0$ for all $t \in I_i^k$. By Nagumo's Theorem, this implies that $\psi_{q-1}^d(\vec{x}) \leq 0$ for all $t \in I_i^k$. Continuing inductively, observe that for all $2 \leq j \leq q$ it holds that $\psi_{j}^d(\vec{x}) \leq 0$ $\forall t \in S_i^k$, which implies $\dot{\psi}_{j-1}^d(\vec{x}) + \alpha_j(\psi_{j-1}^d(\vec{x})) \leq 0$ $\forall t \in I_i^k$. Therefore, by Nagumo's Theorem it holds that $\psi_{j-1}^d(\vec{x}) \leq 0$ $\forall t \in I_i^k$. By this logic we therefore have $\psi_q^d(\vec{x}) \leq 0 \implies \psi_{q-1}^d(\vec{x}) \leq 0 \implies \ldots \implies \psi_1^d(\vec{x}) \leq 0$ $\forall t \in I_i^k$.
By Lemma \ref{lem:getridofdist}, $\psi_1^d(\vec{x}) \geq \dot{\psi}_0^d(\vec{x}) + \alpha_1(\psi_0^d(\vec{x}))$ for all $t \geq 0$. Therefore $\psi_1^d(\vec{x}) \leq 0$ $\forall t \in I_i^k$ implies that $\dot{\psi}_0^d(\vec{x}) + \alpha_1(\psi_0^d(\vec{x})) \leq 0$ $\forall t \in I_i^k$, which implies that $\psi_0^d(\vec{x}) \leq 0$ $\forall t \in I_i^k$.
Using the definitions in \eqref{eq:Spsisets}, it follows that the trajectory $\vec{x}(t)$ satisfies $\vec{x}(t) \in S_I^d = \bigcap_{j=1}^q S_j^d$ for all $t \in I_i^k$, which concludes the proof.
\end{proof}

Under the communication protocol described in Section \ref{sec:asynchronous}, each normal agent $i \in \N$ can use the most recently received inputs $\hat{u}_l^{k_i}$ from other normal agents to calculate a control input $u_i^{k_i} \in K_i(\vec{x}^{k_i})$. Such a $u_i^{k_i}$ can be computed by solving the following QP:
\begin{alignat}{2}
	    &u_i(\vec{x}^{k_i}) = 
		\underset{u_i \in \mathcal{U}_i}{\arg\min}  \nrm{u_i - u_{i,\textup{nom}}^{k_i}}_2 \label{eq:psiQPdistributed} \\
		&\text{s.t.}\quad  \pth{L_{f_i}(\psi_{q-1}^d)^{x_i}(\vec{x}^{k_i}) + L_{g_i}(\psi_{q-1}^d)^{x_i}(\vec{x}^{k_i})u_i} + \nonumber\\
		&  \sum_{l \in \N \backslash \{i\}} \pth{L_{f_l}(\psi_{q-1}^d)^{x_l}(\vec{x}^{k_i}) + L_{g_l}(\psi_{q-1}^d)^{x_l}(\vec{x}^{k_i})\hat{u}_l^{k_i}} + \nonumber\\
		& \hspace{-0.5em} \sum_{j \in \A} \gamma_j^{\max}(\vec{x}^{k_i}) + \alpha(\psi_{q-1}^d(\vec{x}^{k_i})) + \eta(\Gamma_i + \delta^{\max}) \leq 0. \nonumber
\end{alignat}

\subsection{Discussion} 

This section has considered systems satisfying Assumption \ref{assume:relativedeg} where all agents' inputs appear simultaneously for the same relative degree $q$ of $h$ under \eqref{eq:generalsystem}. However, Assumption \ref{assume:relativedeg} may not be satisfied in general for systems composed of agents with heterogeneous control-affine dynamics. A simple example is a system composed of both single- and double-integrator agents with states in $\R^3$. 
Only control inputs for the single integrators appear in the function $\psi_1(\vec{x}, \vec{u})$ from $\eqref{eq:psifuncs}$, while the function $\psi_2(\vec{x}, \vec{u}, \dot{\vec{u}}) = \frac{d}{dt}(\psi_1(\vec{x}, \vec{u})) + \alpha_2(\psi_1(\vec{x}, \vec{u}))$ simultaneously contains single-integrator inputs, time-derivatives of single-integrator inputs, and double-integrator inputs.

The extension of this paper's results to the general case does not immediately follow for two reasons. First, the time derivatives of inputs $\dot{\vec{u}}, \ddot{\vec{u}}, \ldots, \vec{u}^{(r)}$, $r \in \Z_{\geq 1}$ for ZOH controllers are undefined at sampling instances. This necessitates a careful and rigorous mathematical analysis of the behavior of each $\psi_j(\vec{x}, \vec{u}, \dot{\vec{u}},\ldots)$ to ensure that safety can indeed be guaranteed under a ZOH control law.
Second, when considering multi-agent safe set functions $h(\cdot)$ the functions $\psi_j$ for higher values of $j$ are not guaranteed to be convex in $\vec{u}$ when Assumption \ref{assume:relativedeg} is not satisfied. This nonconvexity inhibits the ability to efficiently compute safety-preserving control inputs. We therefore leave the general case as an interesting direction for future investigation.

\section{Simulations}
\label{sec:simulations}

Simulations were performed using a combination of MATLAB and the Julia programming language \cite{Julia-2017}. The simulations used the OSQP optimization package \cite{osqp} and the ForwardDiff automatic differentiation package \cite{RevelsLubinPapamarkou2016}.

While forward invariance of the safe set is guaranteed for any control inputs in the safety-preserving controls sets $K_i(\cdot)$, $K_i^{\psi}(\cdot)$, a key issue is guaranteeing that the sets $K_i(\cdot)$, $K_i^{\psi}(\cdot)$ remain nonempty for all forward time. Due to the difficulty of calculating forward reachable sets for general nonlinear systems subject to disturbances \cite{chen2018decomposition, Liebenwein2018SamplingBasedAA}, prior literature typically does not provide guarantees on the forward nonemptiness of such safety-preserving controls sets except in very specific cases (e.g. when control input constraints are not considered).
Even in the absence of obstacles, it is trivial to find examples where forward invariance of the safe set is impossible in an adversarial setting. Two such examples are given in Figure \ref{fig:infeasible} for single integrator agents in the plane $\R^2$, where adversaries surround a normal agent or pin a normal agent against an obstacle.
Proving the forward nonemptiness of sets $K_i(\cdot)$ and $K_i^{\psi}(\cdot)$, however, is beyond the scope of this paper. 
\begin{figure}
    \centering
    \includegraphics[width=.75\columnwidth]{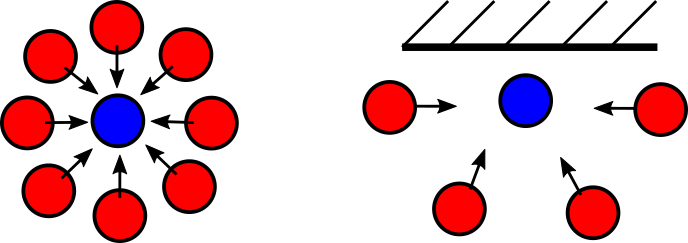}
    \caption{Two examples of initial system states where it is impossible to guarantee forward nonemptiness of the normal agent's safe controls set $K_i(\cdot)$. Agents have single integrator dynamics; the normal agent is depicted in blue and adversarial agents in red. The straight line at the top of the right image denotes an obstacle. Determining initial conditions for which nonemptiness of safe control sets is guaranteed for all forward time remains an open problem when considering nonlinear control-affine systems.}
    \label{fig:infeasible}
\end{figure}

\subsection{Unicycle Agents in $\R^2$}

The first simulation involves a network of $n = 5$ agents with unicycle dynamics in $\R^2$. 
Agents are nominally tasked with tracking time-varying trajectories defined by a Bezier curve, timing law, and local formational offsets. The agents must also avoid static obstacles.
Two agents misbehave by each pursuing the respective closest normal agent. 
The state of each unicycle $i \in \V$ is denoted $x_i = \bmx{x_{i,1}& x_{i,2} & x_{i,3}}^T$. 
Each unicycle is controlled via an input-output linearization method \cite[Ch. 11]{siciliano2010robotics} where each agent has the outputs $p_i = \bmx{p_{i,1} & p_{i,2}}^T$ defined as
\begin{align}
\begin{aligned}
\label{eq:IOoutputs}
    p_{i,1} &= x_{i,1} + b \cos(x_{i,3}), \\
    p_{i,2} &= x_{i,2} + b \sin(x_{i,3}),\ b > 0.
\end{aligned}
\end{align}
The output $p_i$ is treated as having single integrator dynamics $\dot{p}_i  = u_i = \bmx{u_{i,1} & u_{i,2}}^T$.
Each agent $i$ is controlled by first computing the output control input $u_i$ and minimally modifying $u_i$ via the CBF-based QP method described previously. 
The final unicycle control inputs $\bmx{\nu_i & \omega_i}$ are then obtained via the transformation $\bmxs{\nu_i \\ \omega_i} = \bmxs{\cos(\theta_i) & \sin(\theta_i) \\ -sin(\theta_i)/b & \cos(\theta_i)/b} \bmxs{u_{i,1} \\ u_{i,2}}$. At any timestep where the QP is infeasible, each normal agent applies the best-effort safety preserving control \eqref{eq:umin} calculated via the LP \eqref{eq:LPnormal}.
Infeasibility of the QP generating the control inputs does not necessarily imply that safety cannot be maintained.
Reasons why the QP may go infeasible at particular time steps include the conservative nature of the form of $\eta(\cdot)$ and the choice of $\alpha(\cdot)$ function. The LP in \eqref{eq:LPnormal} is applied whenever an agent's QP is infeasible to apply the agent's best control efforts towards maintaining safety.
Given control bounds $|\nu_i| \leq \nu_i^{\max}$ and $|\omega_i| \leq \omega_i^{\max}$, it can be shown that the corresponding linear control bounds on $u_{i,1}, u_{i,2}$ are $A_i(x_i) \bmx{u_{i,1} \\ u_{i,2}} \leq b_i$, with
\begin{align}
\label{eq:IOcontrolbounds}
    \hspace{-0.9em} A_i(x_i) = \bmx{\cos(\theta_i) & \sin(\theta_i) \\ -\cos(\theta_i) & -\sin(\theta_i) \\ -\sin(\theta_i )/ b & \cos(\theta_i )/ b \\ \sin(\theta_i )/ b & -\cos(\theta_i )/ b },\ b_i = \bmx{\nu_i^{\max} \\ \nu_i^{\max} \\ \omega_i^{\max} \\ \omega_i^{\max}}
\end{align}
For strictly positive $\nu_i^{\max}$, $\omega_i^{\max}$, and $b$, the set $\mathcal{U}_i = \{u_i : A_i(x_i) u_i - b_i \leq 0 \}$ satisfies the conditions of Assumption \ref{assume:Ui} for all $x_i \in \R^3$. In this simulation each normal agent has $\nu_i^{\max} = 4$, $\omega_i^{\max} = 2$, $i \in \N$. For purposes of this simulation, each adversarial agent has lower maximum linear and angular velocities than the normal agents with $\nu_j^{\max} = 2$, $\omega_j^{\max} = 1$, $j \in \A$.
The safe set $S$ is defined using a boolean composition of pairwise collision-avoidance sets for normal-to-normal pairs, normal-to-adversarial pairs, and normal-to-obstacle pairs. More specifically, given $i, i' \in \N$ each safe set $h_{i,i'}(\vec{x})$ is defined with respect to the linearized outputs \eqref{eq:IOoutputs} as $h_{i,i'} = (R_c + 2b)^2 - \nrm{p_i - p_{i'}}_2^2$, with partial derivative $\frac{\partial h_{i,i'}}{\partial p_i} = -2(p_{i} - p_{i'})$.
The normal-to-adversarial and normal-to-obstacle pairwise safe sets for $i \in \N$, $j \in \A$ are defined in a similar manner. The pairwise adversarial-to-adversarial and adversarial-to-obstacle safe sets are \emph{not} considered (as per Remark \ref{remark:uncontrollable}), since the nominal control law by definition has no effect on adversarial agents.
All pairwise safe sets are composed into a single CBF $h_{tot}$ via boolean AND operations using the \emph{log-sum-exp} smooth approximation to the $\max(\cdot)$ function:
\begin{align*}
h_{tot}(\vec{x}) = \textup{LSE}(\bmx{h_1,\ldots, h_p}) &= \sigma + \frac{1}{\rho}\ln \pth{\sum_{i=1}^p e^{\rho (h_i - \sigma)}},
\end{align*}
where $\rho \in \R_{>0}$, $\sigma \in \R$.
The term $\sigma$ is chosen to ensure numerical stability.
The term $\rho$ controls how tightly $\LSE(\cdot)$ approximates $\max(\cdot)$.
The reader is referred to \cite{lindemann2018control}, \cite[Eq (10)]{wurts2020collision} for more details.
Sampling times in this simulation are asynchronous; each agent has a nominal sampling time period of $\Gamma = 0.01$ with a time-varying random disturbance satisfying $\delta_i^{\max} = .002$. For each agent $i \in \V$, the disturbance bound satisfies $\phi_i^{\max} = 1.73$, and the term $\eta$ is set as $\eta(\Gamma) = 8.0566$.
Several frames from the simulation are shown in Figure \ref{fig:Sim1}. A plot of $h_{\textup{tot}}$ is given in Figure \ref{fig:htot_value}. As shown by Figure \ref{fig:htot_value}, under the proposed resilient controller the safety bounds for normal agents are not violated for the duration of the simulation. This is achieved despite the actions of the adversarial agents.

For comparison, Figure \ref{fig:htot_value_2} depicts a simulation run under the same parameters but with $\eta(\Gamma) = 0$ $\forall t \geq 0$; i.e. nothing is done by normal agents to counteract effects of sampling, disturbances, and time delays. In this case safety of the normal agents is not preserved---the value of $h_{\text{tot}}$ is temporarily positive, indicating that one or more of the composed safe sets was not invariant for the entire simulation.

\begin{figure*}
    \centering
    \includegraphics[width=\figsize\textwidth]{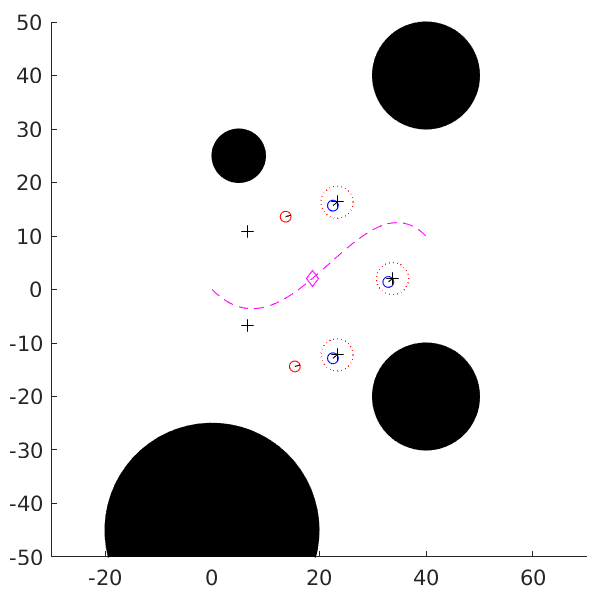} 
    \includegraphics[width=\figsize\textwidth]{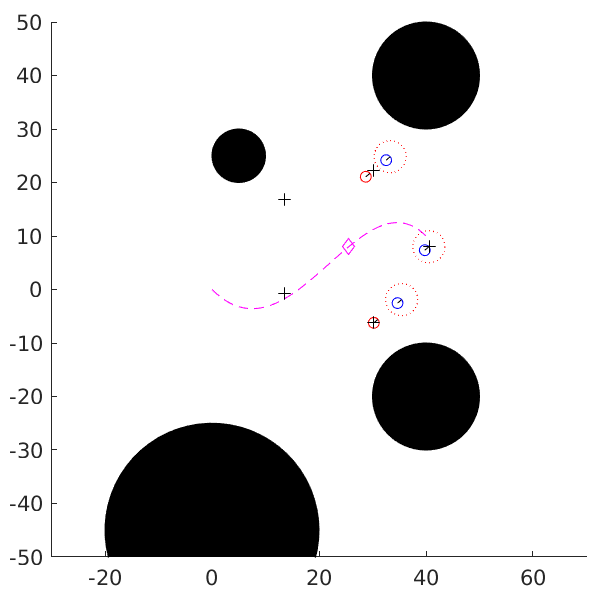} 
    \includegraphics[width=\figsize\textwidth]{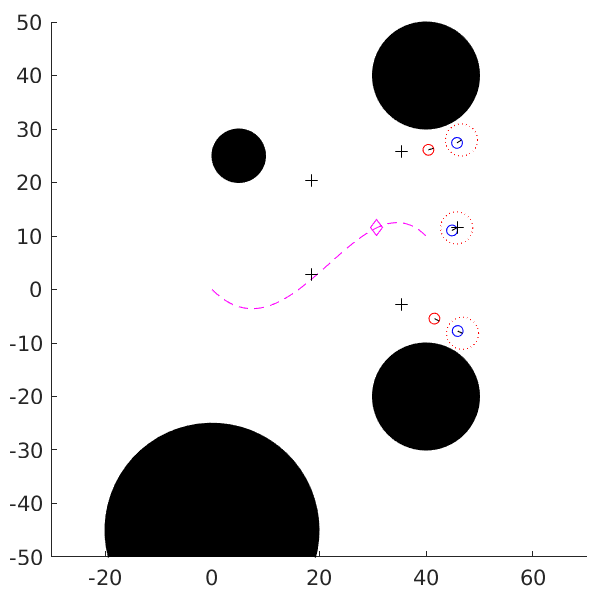} 
    
    
    \caption{Still frames from the video of Simulation 1. Normal agents are represented by blue circles and adversarial agents are represented by red circles. The dotted red lines around the blue circles represent normal agents' safety radii. The time-varying formation trajectory is represented by the dotted magenta line; the magenta diamond represents the center of formation. Black crosses represent agents' nominal local time-varying formational points.}
    \label{fig:Sim1}
\end{figure*}

\begin{figure}
    \centering
    \includegraphics[width=.75\columnwidth]{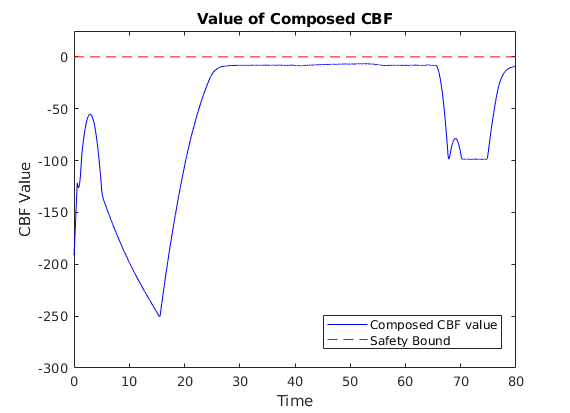}
    \caption{The value of the composed function $h_{tot}$ representing the safe set $S$. Non-positive values represent safety of the normal agents.}
    \label{fig:htot_value}
\end{figure}

\begin{figure}
    \centering
    \includegraphics[width=.75\columnwidth]{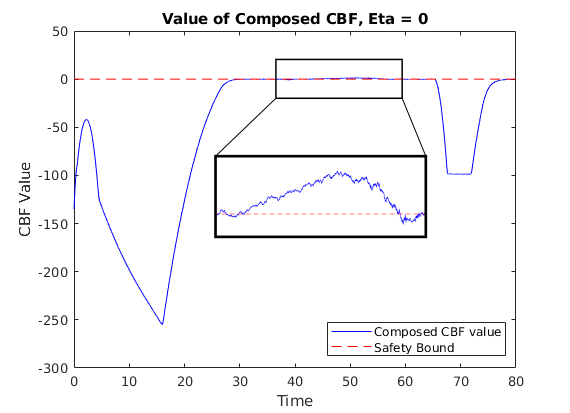}
    \caption{The value of the composed function $h_{tot}$ representing the safe set $S$ when $\eta(\Gamma) = 0$ for all normal agents; i.e. sampling times and disturbances are not accounted for in the control input calculations. The safety bound for the normal agents is violated.}
    \label{fig:htot_value_2}
\end{figure}

\begin{figure}
    \centering
    \includegraphics[width=.75\columnwidth]{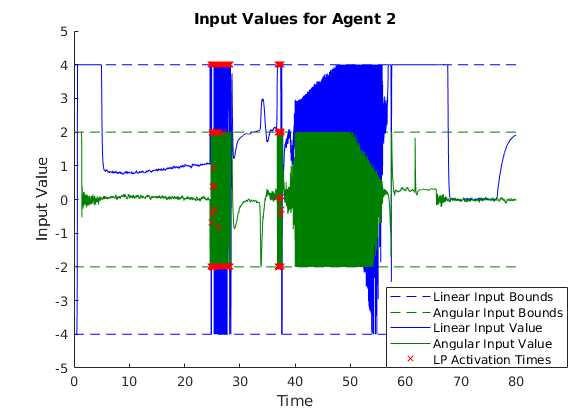}
    \caption{Input values for (normal) agent 2. The blue solid line represents linear input value and the green solid line represents angular input value. Dotted lines represent input bounds. Times at which the worst-case LP is used are marked with red X's on both the linear and angular input lines.}
    \label{fig:my_label}
\end{figure}

\subsection{Double Integrators in $\R^3$}

The second simulation involves a network of $n = 8$ double integrator agents in $\R^3$. 
Four of the agents behave normally and four are adversarial.
Similar to the prior simulation,
agents are nominally tasked with tracking positions in a time-varying formation defined by a Bezier curve, timing law, and local formational offsets. Each agent $i \in \V$ has the state $\vec{x}_i = \bmx{x_{i,1} & x_{i,2} & x_{i,3} & v_{i,1} & v_{i,2}& v_{i,3}}^T$ with the following dynamics:

\begin{align*}
    \dot{\vec{x}}_i &= \underbrace{\bmx{\bm 0_{3 \times 3} & I_{3\times 3} \\ \bm 0_{3 \times 3} & -\beta_i I_{3 \times 3}}}_{A} \vec{x}_i + \underbrace{\bmx{\bm 0_{3 \times 3} \\ I_{3 \times 3}}}_{B} \bmx{u_{i,1} \\ u_{i,2} \\ u_{i,3}} + \phi_i(t).
\end{align*}
Each normal agent $i \in \N$ has an input bound $\nrm{u_i}_{\infty} \leq 2$. Each adversarial agent $j \in \A$ has an input bound $\nrm{u_j}_\infty \leq 1.5$. 
The terms $\beta_i \in \R_{\geq 0}$ are chosen such that each normal agent has a velocity bound $\nrm{\bmx{v_{i,1} & v_{i,2} & v_{i,3}}} \leq 3$ $\forall i \in \N$ and each adversarial agent has $\nrm{\bmx{v_{j,1} & v_{j,2} & v_{j,3}}} \leq 2.25$ $\forall j \in \A$. Specifically, $\beta_i =  3$ $\forall i \in \V$.

Each normal agent $i \in \N$ seeks to track a time-varying formational state $\vec{x}^d_i \in \R^3$. 
The nominal formation states for all agents are equidistantly distributed around the edge of a circle of radius 30 whose center translates along a time-varying trajectory described by a 3rd order Bezier curve $B(t) = \sum_{k=0}^3 \vec{\beta}_i b_{i,3}(s(t))$ described by the timing law $s(t) = \frac{t_f - t}{t_f - t_0}$ for $t_f = 140$ and $t_0 = 0$, Bernstein basis polynomials $b_{i,3}(s)$, and the vector coefficients
\begin{align*}
    \vec{\beta}_0 = \bmx{0 \\0\\0}\ \vec{\beta}_1 = \bmx{-25\\25\\30}\ \vec{\beta}_2 = \bmx{125\\75\\-30}\ \vec{\beta}_3 = \bmx{100\\100\\0}.
\end{align*}
Letting the error $\vec{e}_i$ be defined as $\vec{e}_i = \vec{x}^d_i - \vec{x}_i$,
each $i \in \N$ calculates the nominal control law $\vec{u}_{i,\text{nom}} = -K\vec{e}_i - \ddot{\vec{x}}^d_i$ with $K = \bmxs{k_1 I_{3\times 3} & k_2 I_{3 \times 3}}$, where $\ddot{\vec{x}}^d_i$ is the acceleration of $\vec{x}^d_i$, $k_1 = 2$, and $k_2 = 2 \sqrt{k_1}$. 
The nominal input $\vec{u}_{i,\text{nom}}$ is minimally modified via the higher-order CBF-based QP method described in \ref{sec:highdegree}. 
Similar to \eqref{eq:LPnormal},
at any timestep $t_i^k$ where the QP is infeasible each normal agent $i \in \N$ applies the control action
\begin{align*}
    u_i^{\min}(\vec{x}^{k_i}) = \arg\min_{u_i \in \mathcal{U}_i} \bkt{L_{f_i}\psi_{q-1}^{x_i}(\vec{x}^{k_i}) + L_{g_i}\psi_{q-1}^{x_i}(\vec{x}^{k_i}) u_i}.
\end{align*}

The environment contains 10 spherical obstacles with radius 2 randomly distributed across the volume containing the second half of the time-varying trajectory.
Adversarial agents $j \in \A$ in this simulation are each assigned a target agent to pursue, with one of the normal agents having multiple pursuers. Each adversarial agent $j \in \A$ is assumed to have full knowledge of its target's current state, but does not have knowledge of its target's control inputs. 
Defining the error term $\vec{e}_{i,j}  = \vec{x}_i - \vec{x}_j$, $i,j \in \V$,
each adversary $j \in \A$ applies the control law 
$\vec{u}_j = -K\vec{e}_{i,j}$, 
where the matrix $K$ is defined as previously described but with $k_1 = 1$.
This control input is minimally modified using a CBF QP method to respect control input constraints and avoid collisions with other adversaries and obstacles, but not with normal agents.

The safe set $S$ in this simulation is defined using a similar boolean composition of pairwise collision avoidance sets as in the previous simulation. At each sampling instance, the normal agent $i$ considers all other agents whose positions lie within a neighborhood of radius $35$ from agent $i$'s position $\bmxs{x_{i,1} & x_{i,2} & x_{i,3}}$. All normal-to-normal, normal-to-adversarial, and normal-to-obstacle pairwise safe sets are composed into a single function $h_{tot}$ via boolean AND operations using the \emph{log-sum-exp} function.
Sampling times in this simulation are asynchronous for normal agents; each $i \in \N$ has a nominal sampling time period of $\Gamma = 0.07$ with $\delta_i^{\max} = 0.03$ for each normal agent.
The disturbance $\phi_i(t)$ for each agent $i \in \V$ (normal and adversarial) satisfies $\phi_i^{\max} = .4899$. For each normal agent $i \in \N$ the term $\eta$ satisfies $\eta(\Gamma_i + \delta^{\max}) = 5$, and the term $\xi$ satisfies $\xi = 39.19$.
Still frames from the simulation are shown in Figure \ref{fig:Sim2}, and a plot of the value of $h_{tot}$ is given in Figure \ref{fig:htot3D}. As shown by Figure \ref{fig:htot3D}, the safety bounds for the normal agents are not violated for the duration of the simulation despite the actions of the adversaries.

\begin{figure*}
    \centering
    \includegraphics[width=\figsize\textwidth]{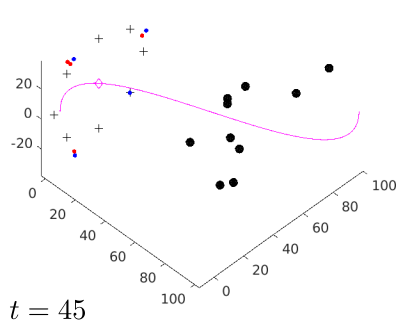} 
    \includegraphics[width=\figsize\textwidth]{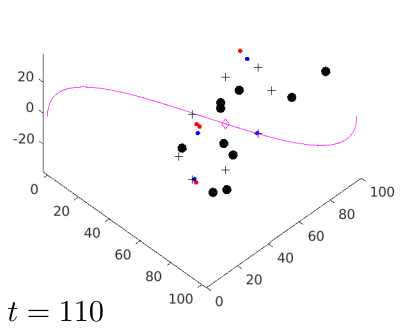}
    \includegraphics[width=\figsize\textwidth]{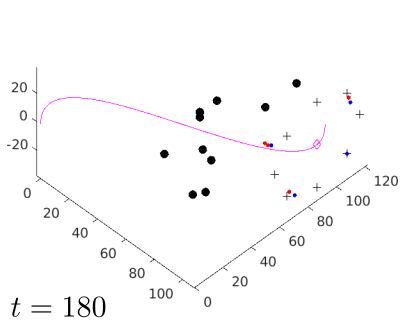} 
    
    
    \caption{Still frames from Simulation 2. Normal agents are represented by blue circles and adversarial agents are represented by red circles. For clarity, the safety radii of the normal agents has been omitted. The time-varying formation trajectory is represented by the dotted magenta line; the magenta diamond represents the center of formation. Black crosses represent individual agents' nominal local time-varying formational points. Black spheres represent randomly placed obstacles.}
    \label{fig:Sim2}
\end{figure*}

\begin{figure}
    \centering
    \includegraphics[width=0.75\columnwidth]{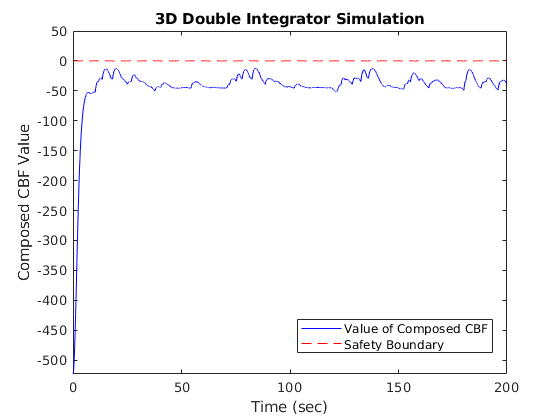}
    \caption{A plot of the value of the composed function $h_{tot}$ representing the safe set $S$ for all normal agents in the second simulation. Non-positive values represent safety of the normal agents. For the entire duration of this simulation, the value of $h_{tot}$ remains strictly negative, indicating that safety is maintained for all normal agents.}
    \label{fig:htot3D}
\end{figure}

\begin{figure}
    \centering
    \includegraphics[width=0.75\columnwidth]{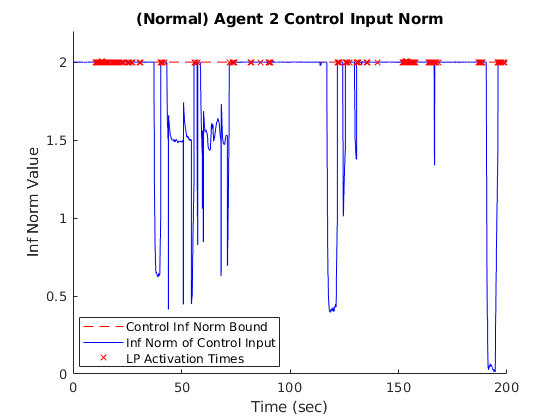}
    \caption{A plot of the infinity norm of control input for (normal) agent 2 versus time for the second simulation. The control norm bound is plotted in red, and the norm of agent 2's control input is plotted in blue. Times when the backup LP is used are marked with red X's.}
    \label{fig:my_label}
\end{figure}

\section{Conclusion}
\label{sec:conclusion}

In this paper, we presented a framework for normally-behaving agents to render a safe set forward invariant in the presence of adversarial agents. The proposed method considers distributed sampled-data systems with heterogeneous, asynchronous control affine dynamics, and a class of functions defining safe sets with high relative degree with respect to system dynamics.
Directions for future work include investigating 
cases where control inputs of heterogeneous agents do not appear simultaneously in higher derivatives of the functions describing safe sets.



\bibliographystyle{IEEEtran}
\bibliography{bibliography.bib}

\end{document}